\documentclass[oneside]{amsart}
\usepackage{amssymb}
\usepackage{amsthm}
\usepackage{amsmath}
\usepackage{amsthm}
\usepackage{amsfonts}
\usepackage{pb-diagram}
\usepackage{color}
\usepackage[normalem]{ulem}
\usepackage{tikz}
\usepackage{hyperref}
\usepackage{cleveref}
\usepackage{nameref}
\usepackage{enumerate}
\usepackage{xcolor}
\usepackage{parskip}
\usepackage{thmtools}
\usepackage{thm-restate}
\usepackage{geometry}
\usepackage{cancel}

\makeatletter
\def\namedlabel#1#2{\begingroup
    #2%
    \def\@currentlabel{#2}%
    \phantomsection\label{#1}\endgroup
}

\hypersetup{
  colorlinks   = true, 
  urlcolor     = {blue!50!black}, 
  linkcolor    = {blue!50!black}, 
  citecolor   = {red!50!black} 
}

\newtheorem{theorem}{Theorem}[section]
\newtheorem*{theorem*}{Theorem}
\newtheorem{proposition}[theorem]{Proposition}
\newtheorem{lemma}[theorem]{Lemma}
\newtheorem{corollary}[theorem]{Corollary}

\newtheorem{fact}[theorem]{Fact}
\newtheorem{question}[theorem]{Question}

\newtheorem{thmx}{Theorem}

\theoremstyle{definition}
\newtheorem{definition}[theorem]{Definition}
\newtheorem{example}[theorem]{Example}

\theoremstyle{remark}
\newtheorem{remark}[theorem]{Remark}
\newtheorem{notation-num}{Notation}
\newtheorem*{notation}{Notation}

\newcommand{\R}{\mathbb{R}}
\newcommand{\F}{\mathcal{F}}
\newcommand{\Scal}{\mathcal{S}}
\newcommand{\hull}{\mathsf{hull}}

\newcommand{\VC}{\mathsf{VC}\text{-}\mathsf{dim}}
\newcommand{\pow}{\mathcal{P}}
\newcommand{\collin}{\mathsf{collin}}

\newcommand{\sm}{\setminus}
\newcommand{\es}{\emptyset}

\newcommand{\sub}{\subseteq}

\newcommand{\CF}{{\mathcal F}}

\newcommand{\Abb}{\ensuremath{\mathbb{A}}}

\newcommand{\Nbb}{\ensuremath{\mathbb{N}}}

\newcommand{\Rbb}{\ensuremath{\mathbb{R}}}

\title{Exact VC-Dimensions of Certain Geometric Set Systems} 

\subjclass[2020]{Primary: 52C35. Secondary: 03C98.}
\keywords{VC-dimension, $k$-fold unions, shatter isomorphisms}

\author[P.E. Eleftheriou] {Pantelis  E. Eleftheriou}
\address{School of Mathematics, University of Leeds, Leeds LS2 9JT, United Kingdom}
\email{\href{mailto:P.Eleftheriou@leeds.ac.uk}{P.Eleftheriou@leeds.ac.uk}}

\author[A. Papadopoulos ]{Aris Papadopoulos}
\address{Department of Mathematics, University of Maryland, College Park}
\email{\href{mailto:aris@umd.edu}{aris@umd.edu}}

\author[F. Westhead ]{Francis Westhead}
\address{Department of Mathematics, University of Maryland, College Park}
\email{\href{mailto:westhead@umd.edu}{westhead@umd.edu}}

\thanks{P.E. was partially supported by an EPSRC Early Career Fellowship (EP/V003291/1). A.P. was partially supported by a Leeds Doctoral Scholarship from the University of Leeds. F.W. was partially supported by an LMS undergraduate research bursary (URB-2022-54), a Dean's Fellowship from the University of Maryland, and NSF grant DMS-2154101.}

\date{\today}

\begin{document}

\begin{abstract} 
The VC-dimension of a family of sets is a measure of its combinatorial complexity used in machine learning theory, computational geometry, and even model theory. Computing the VC-dimension of the $k$-fold union of geometric set systems has been an open and difficult combinatorial problem, dating back to \cite{BEHWM1989}, who ask about the VC-dimension of $k$-fold unions of half-spaces in $\R^d$.

Let $\F_1$ denote the family of all lines in $\R^2$. It is well-known that $\VC(\F_1) = 2$. In this paper, we study the $2$-fold and $3$-fold unions of $\F_1$, denoted $\F_2$ and $\F_3$, respectively. We show that $\VC(\F_2) = 5$ and $\VC(\F_3) = 9$. Moreover, we give complete characterisations of the subsets of $\R^2$ of maximal size that can be shattered by $\F_2$ and $\F_3$, showing they are exactly two and five, respectively, up to isomorphism in the language of the point-line incidence relation.
\end{abstract}

\maketitle

\section{Introduction}\label{sec:intro}

The \emph{Vapnik-Chervonenkis dimension} (\emph{VC-dimension}) of a family of sets is a well-established measure of ``combinatorial complexity'' which has found immense applications in statistical learning theory, machine learning, computational geometry, and model theory. The notion of VC-dimension was introduced in \cite{VapnikChervonenkis1971} as a condition in statistical learning. One of the fundamental results about families with finite VC-dimension is the celebrated \emph{VC-theorem}, which is roughly a uniform version of the weak law of large numbers for families of subsets of a probability space. See \cite[Theorem~$6.6$]{Simon2015} for a modern exposition of the VC-theorem. Independently, and around the same time as Vapnik and Chervonenkis, Shelah introduced in \cite{Shelah1971} a notion of model-theoretic tameness, which he called \emph{NIP} (\emph{Not the Independence Property}). Roughly twenty years later, Laskowski observed, in \cite{Laskowski1992}, that the model-theoretic notion of NIP for a definable family of sets corresponds precisely to the finiteness of its VC-dimension.

In the years following the works of both Vapnik and Chervonenkis, and Shelah, the notion of VC-dimension has found wide application in machine learning. In the seminal paper \cite{BEHWM1989}, the authors argue that the finiteness of VC-dimension is the ``essential condition for distribution-free learning''. On the other hand, in model theory, NIP has turned out to be one of the core notions of tameness and has been the focus of a lot of work in recent years. Simon's book \cite{Simon2015} gives a good account of recent developments in the study of NIP theories.

Let us now formally introduce the notion of VC-dimension. First of all, by a \emph{set system} we mean a pair $(X,\Scal)$, where $X$ is a (possibly infinite) set and $\Scal\subseteq\pow(X)$ is a family of subsets of $X$. Given a set system $(X,\Scal)$ and $Y\subseteq X$, we say $\Scal$ \emph{shatters} $Y$ if for each $Z\subseteq Y$ there is $S\in \mathcal{S}$ such that $S\cap Y= Z$. More generally, we say that some $Z\subseteq Y$ is \emph{isolated} by $\mathcal{S}$ if there is $S\in \mathcal{S}$ such that $S\cap Y=Z$. 
In particular, $Y\subseteq X$ is shattered by $\mathcal{S}$ if and only if every $Z\subseteq Y$ is isolated by $\mathcal{S}$.   The \emph{VC-dimension} of a set system $(X,\mathcal{S})$ is defined to be the largest positive integer $n$ for which there is some $Y\subseteq X$ of size $n$ that is shattered by $\mathcal{S}$, if one exists. In this case, we write $\VC((X, \mathcal{S})) = n$, or more simply, $\VC(\mathcal S)=n$, if $X$ is clear from the context. It could well be the case that for all $n\in\mathbb{N}$, there is some $Y\subseteq X$ of size $n$ which is shattered by $\mathcal{S}$. Then, we say that $\mathcal{S}$ has infinite VC-dimension, and write $\VC(\mathcal{S}) = \infty$. In the context of our paper, all families have (unless otherwise stated) finite VC-dimension. 

In this paper, we are interested in calculating the exact VC-dimension of certain set systems which arise as \emph{$k$-fold unions} of simpler ones. To fix notation, if $(X,\mathcal{S})$ is a set system, then we define the \emph{$k$-fold union of $\mathcal{S}$}, denoted $\mathcal{S}^{\cup k}$, as follows
\[
    \mathcal{S}^{\cup k} := \left\{\bigcup_{i=1}^k S_i : S_i\in\mathcal{S}\right\}.
\]
In particular, $\mathcal{S}\subseteq\mathcal{S}^{\cup k}$.

Let $\F_1$ denote the family of all lines in $\R^2$. To simplify notation, we denote $\F_1^{\cup k}$ by $\F_k$, for $k\in\mathbb{N}_{\geq 1}$. In this notation, our first main result is the following theorem.

\begin{thmx}[\Cref{vc-2-lines,vc-3-lines}]\label{thm:vc-intro}
    $\VC(\F_2)= 5$ and $\VC(\F_3) = 9$.
\end{thmx}

Although the arguments in our paper do not use any heavy mathematical machinery and are essentially just elementary combinatorial arguments, the approach we take towards proving \Cref{thm:vc-intro} is novel. This is because we tackle the problem of calculating the VC-dimension of $\F_2$ and $\F_3$ by axiomatising the necessary and sufficient conditions for a set of points to be shattered by $\F_2$ and $\F_3$, respectively. 

In particular, the only essential geometric property of lines in $\R^2$ we use is the \emph{unique line property}, that is, given any two distinct points $a,b\in\R^2$ there is a unique line passing through both of them. Therefore, the conditions we give in \Cref{sec:vc-2-lines} and \Cref{sec:vc-3-lines} (or possibly adapted versions of these conditions) could be applied in other contexts to compute the VC-dimension of families with the unique line property. In \Cref{sec:higher-arity}, we illustrate this idea by showing that our result transfers essentially to shattering $(n-2)$-dimensional subspaces of $\R^n$ by $(n-1)$-dimensional subspaces (see \Cref{thm:linear-alg}, for a more precise formulation). 

At the heart of our paper is the new notion of a \emph{shatter-isomorphism} (see \Cref{def:shatter-isomorphic} for the precise statement), which we introduce in \Cref{sec:shatter-isomorphism}. This abstract notion is powerful enough to allow us to: (a) compare the complexity of set systems on different sets; and (b) examine the finer structure of shatterable sets. The proofs of the already mentioned results of \Cref{sec:higher-arity} illustrate (a). To illustrate (b), in \Cref{sec:counting}, we use the axiomatisations obtained in Sections \ref{sec:vc-2-lines} and \ref{sec:vc-3-lines} to completely classify (up to shatter-isomorphism) the sets of size $5$ and $9$ which are shattered by $\F_2$ and $\F_3$, respectively. This culminates in the following result.

\begin{thmx}[\Cref{iso2,iso3}]\label{thm:classification-intro}
    There are, up to shatter-isomorphism, exactly two non-isomorphic sets of $5$ points in $\R^2$ that can be shattered by $\F_2$ and exactly five non-isomorphic sets of $9$ points in $\R^2$ that can be shattered by $\F_3$. 
\end{thmx}

We also introduce the notion of a \emph{maximal-shattering sequence}, denoted $(s_k(X,\mathcal{S}))_{k\in\mathbb{N}}$, which for a fixed set-system $(X,\mathcal{S})$ of finite VC-dimension counts the number of non-isomorphic (in the sense of shatter-isomorphism) structures of size $\VC(\mathcal{S}^{\cup k})$ which can be shattered by $\mathcal{S}^{\cup k}$. In this notation, \Cref{thm:classification-intro} says that $s_2(\mathcal{F}_1) = 2$ and $s_3(\mathcal{F}_1) = 5$. We leave the computing $s_k(\F_1)$ for $k\geq 4$, and determining general properties of the sequence $(s_k)_{k\in\mathbb{N}}$ as open problems.

\textbf{Structure of the paper.} Our paper is structured as follows. In \Cref{sec:background}, we give a short overview of some results in the VC-dimension of $k$-fold unions of families. In \Cref{sec:shatter-isomorphism}, we briefly introduce the notion of shatter isomorphism and discuss some easy observations. We then move on to our geometric set systems. In \Cref{sec:terminology-notation} we collect the main definitions that are needed to prove \Cref{thm:vc-intro}. Then, in \Cref{sec:vc-2-lines,sec:vc-3-lines}, we calculate the VC-dimension of $\F_2$ and $\F_3$, respectively. In \Cref{sec:counting}, we prove \Cref{thm:classification-intro}. Finally, in \Cref{sec:higher-arity} we deduce higher-dimension analogues of \Cref{thm:vc-intro}.

\section{Brief Historical Remarks}\label{sec:background}

In this paper, we are interested in \emph{exact} values of VC-dimension of $k$-fold unions. That being said, there has, since the late 1980s, been significant work carried out determining the \emph{asymptotic} growth of $\VC(\Scal^{\cup k})$. In \cite{BEHWM1989}, among other results, the authors provide general upper and lower bounds for the VC-dimension of $\mathcal{S}^{\cup k}$, in terms of $k$ and $\VC(\mathcal{S})$. More precisely, they show the following statement.

\begin{fact}[{\cite[Lemma~3.2.3]{BEHWM1989}}]\label{thm:union-bounds}
    Let $k\in\mathbb{N}$ and $(X,\mathcal{S})$ be any set system. Then:
    \[
    \VC(\mathcal{S}^{\cup k}) = O\left(\VC(\mathcal{S})\cdot k\log k\right).
    \]
    Moreover, there are set systems $(X,\mathcal{S})$ such that $\VC(\mathcal{S}^{\cup k}) = \Omega(\VC(\mathcal{S})\cdot k)$.
\end{fact}

Hence, there is a very clear gap between the known lower and upper bounds for $\VC(\mathcal{S}^{k})$ and in \cite{BEHWM1989} the authors ask whether the upper bound they obtain is tight in certain geometric cases. In \cite{Eisenstat2007} it is shown via a probabilistic construction that for all $d\geq 5$ there exist abstract set systems $(X,\mathcal{S})$ such that $\VC(\mathcal{S})=5$, obtaining the upper bound in \Cref{thm:union-bounds}. Moreover, they show that if $\VC(\mathcal{S}) = 1$ then $\VC(\mathcal{S}^{\cup k})$ is bounded above by $k$, which is, of course also tight. The gap for set systems $(X,\mathcal{S})$ with $\VC(\mathcal{S})\in\{2,3,4\}$ was filled in by \cite{Eisenstat2009}, where it is shown that there exist abstract set systems with VC-dimension in $\{2,3,4\}$ witnessing that the upper bound of \Cref{thm:union-bounds} is tight.

The study of the VC-dimension of geometric set systems was initiated in \cite{BEHWM1989}. Recently, in \cite{CMK2019}, the authors prove that in the case of the $k$-fold unions of half spaces in $\R^d$, for any $d\geq 4$ the upper bound from \cite{BEHWM1989} is indeed tight, resolving that problem after thirty years. Their argument makes use of some new results concerning the existence of $\epsilon$-nets (more precisely \cite{Kupavskii2016}), and we refer the reader to \cite{CMK2019} for a discussion of VC-dimension in computational geometry.

\section{The Notion of Shatter Isomorphism}\label{sec:shatter-isomorphism}

In this relatively short section, we introduce the notion of \emph{shatter structures} and \emph{shatter isomorphisms}. We use these notions in \Cref{sec:counting} to classify the number of maximal (in cardinality) sets of points shattered by unions of $2$ and $3$ lines, and in \Cref{sec:higher-arity} to explain how our results on unions of points and lines on the plane can be translated to results about unions of affine subsets of $n$-dimensional vector spaces. The reader who wants to delve directly into the proof of \Cref{thm:classification-intro} can move on to the next section, and consult the terminology introduced below before reading \Cref{sec:counting,sec:higher-arity}.

Let $(X,\mathcal{S})$ be a set system, with $X$ not necessarily finite. Naturally, this induces a set system on any $P\subseteq X$. We are primarily interested in $P\subseteq X$ such that $\Scal$ shatters $P$. To avoid cardinality issues (and trivialities which may arise from them), we want to consider, for any $P\subseteq X$ a ``minimal subfamily'' $\Scal(P)\subseteq \Scal$ such that $(P,\Scal(P))$ captures the combinatorial complexity of $(P,\Scal)$. To this end, we define an equivalence relation $\sim_P$ on $\Scal$ as follows:
\[
S\sim_P S'\text{ if and only if }P\cap S = P\cap S',
\]
for all $S,S'\in \Scal$.

Let $\Scal(P)\subseteq \Scal$ consist of representatives for the equivalence classes in $\Scal/_{\sim_P}$, and let $\mathcal{L}_{\mathcal{S}(P)}$ be the language consisting of a unary predicate $S_i$ for each $S_i\in\mathcal{S}(P)$. We may view $(P,\mathcal{S}(P))$ as a first-order $\mathcal{L}_{\mathcal{S}(P)}$-structure, $\mathcal{P}_\mathcal{S}$, as follows. The domain of $\mathcal{P}_\mathcal{S}$ is $P$, and we interpret each $S\in\mathcal{L}_{\mathcal{S}(P)}$ as follows:
\[
    \mathcal{P}_\Scal\vDash S(x) \text{ if and only if } x\in S.
\]
We call $\mathcal{P}_\mathcal{S}$ the \emph{set-system structure} of $(P,\mathcal{S})$.

\begin{example}\label{ex:lines-points}
    Let $L$ be the set of all lines in $\Rbb^2$ and  $P = \{p_0,\dots,p_{n-1}\}\subseteq \R$. Then, $\mathcal{L}_{L(P)}$ consists of all lines which intersect (at least) two points from $P$ together with, for each $p\in P$ a unique line $l\in L$ such that $l\cap P = \{p\}$.
\end{example}

To simplify exposition, given two set system structures $A_{\Scal}=(A,\mathcal{S})$ and $B_{\Scal'}=(B,\mathcal{S}')$ such that $|\Scal(A)| = |\Scal'(B)|$, that is, $|\mathcal{L}_{\mathcal{S}(A)}| = |\mathcal{L}_{\mathcal{S}'(B)}|$, we may relabel predicates in the two languages, so that $A_{\mathcal{S}}$ and $B_{\mathcal{S}'}$ can be viewed as $\mathcal{L}$-structures, in the same language $\mathcal{L}$. Since we always work up to relabelling of set-system languages, the choice or representatives discussed in the previous paragraph does not affect any of the definitions to come.

This leads us to the main definition of this section.

\begin{definition}\label{def:shatter-isomorphic}
    We say that two set-system structures $(A,\mathcal{S})$ and $(B,\mathcal{S}')$ are \emph{shatter-isomorphic} if $|\mathcal{L}_{\mathcal{S}(A)}| = |\mathcal{L}_{\mathcal{S}'(B)}|$ and there is some relabelling of the predicates in such that $A_{\mathcal{S}}$ and $\mathcal{B}_{\mathcal{S}'}$ are isomorphic, as structures in the same language $\mathcal{L}$.    
\end{definition}

The following remark is immediate.

\begin{remark}
    If $(X_1,\mathcal{S}_1)$ and $(X_2,\mathcal{S}_2)$ are shatter-isomorphic then $\mathcal{S}_1$ shatters $X_1$ if and only if $\mathcal{S}_2$ shatters $X_2$.
\end{remark}

Recall that given a set system $(X,\mathcal{S})$ and an integer $k\in\mathbb{N}$ we write $\mathcal{S}^{\cup k}$ for the family of $k$-fold unions of sets in $\mathcal{S}$. The following is an immediate consequence of the definitions.

\begin{lemma}\label{lem:k-fold-unions}
    Suppose $(X_1,\mathcal{S}_1)$ and $(X_2,\mathcal{S}_2)$ are two shatter-isomorphic set systems. Then, for all $k\in \mathbb{N}$, $X_1$ is shattered by $\mathcal{S}_1^{\cup k}$ if and only if $X_2$ is shattered by $\mathcal{S}_2^{\cup k}$.
\end{lemma}

Fix an infinite set $X$, and let $\mathcal{S}\subseteq \mathcal{P}(X)$ be a family of subsets of $X$ such that $\VC(\mathcal{S})<\infty$. Since $\VC(\mathcal{S})<\infty$, it follows that $\VC(\mathcal{S}^{\cup k})<\infty$ for each $k\in\mathbb{N}$. Let $d_k(X,\mathcal{S}) = \VC(\mathcal{S}^{\cup k})$. We write $s_k(X,\mathcal{S})$ for the number of non-shatter isomorphic set-system structures, $(Y,\mathcal{S})$, such that $Y\subseteq X$ has size $d_k(X,\mathcal{S})$ and $\mathcal{S}^{\cup k}$ shatters $Y$. That is, $s_k(X,\mathcal{S})$ counts the number of isomorphism-types of maximal subsets of $X$ which can be shattered by the $k$-fold union of $\mathcal{S}$. We call $(s_k(X,\Scal))_{k\in\Nbb}$ the \emph{maximal-shattering sequence} of $(X,\Scal)$. When $(X,\mathcal{S})$ is understood, we may write just $d_k$ and $s_k$, in place of $d_k(X,\Scal)$ and $s_k(X,\Scal)$, respectively.

Clearly, given $(X,\mathcal{S})$, $(d_k)_{k\in\mathbb{N}}$ is non-decreasing, since $\mathcal{S}\subseteq\mathcal{S}^{\cup k}$. But it is easy to see that it need not be strictly increasing. 

\begin{example}
    Let $\mathcal{S}=\{(a,b):a<b\in\mathbb{R}\}$, be the family of all open intervals in $\mathbb{R}$. It is an easy exercise to show that $\VC(\mathcal{S}^{\cup k}) = 2k$, for all $k\in\mathbb{N}$. It is also not too hard to see that $s_k = 1$, for all $k\in\mathbb{N}$. Indeed, given $k\in\mathbb{N}$, all sets of size $2k$ can be shattered by $\mathcal{S}^{\cup k}$, and all such sets are shatter-isomorphic since they are isomorphic as linear orders.
\end{example}

\begin{lemma}
    If $\VC(\Scal)<\infty$ and $\Scal$ is closed under arbitrary intersections, then $s_1(X,\Scal)=1$.
\end{lemma}
\begin{proof}
    Since $\Scal$ is closed under arbitrary intersections, given $Y\subseteq X$, we define the \emph{$\Scal$-hull} of $Y$ to be:
    \[
        \Scal\text{-}\hull(Y) := \bigcap_{S\in\Scal, Y\subseteq S} S.
    \]
    In this case, it is easy to see that if $\Scal$ shatters $Y\subseteq X$, then for any $Y_0\subseteq Y$ we have that $\Scal\text{-}\hull(Y_0)$ is a representative for the equivalence class of all elements $S\in \Scal$ such that $Y\cap S = Y_0$. In particular, $\mathcal{L}_{\Scal(Y)} =  \{\Scal\text{-}\hull(Y_0):Y_0\subseteq Y\}$.
    
    Now, let $d=\VC(\Scal)$ and suppose that $\Scal$ shatters $Y = \{y_1,\dots,y_d\}$ and $Y' = \{y_1',\dots,y_d'\}$. We have that $\mathcal{L}_{\Scal(Y)} = \mathcal{L}_{\Scal(Y')}$, since they both consist of a single representative for the $\Scal$-hull of each subset of $Y$ and $Y'$, respectively. Making appropriate identifications between the points in $Y$ and $Y'$ gives the required result.
\end{proof}

\begin{remark}
    The condition that $\Scal$ is closed under (arbitrary) intersections in the previous lemma is not necessary for $s_1(X,\Scal) = 1$. Indeed, in the notation of \Cref{ex:lines-points}, we have that $s_1(P,L) = 1$, but $L$ is not closed under intersections.
\end{remark}

\begin{question}
    Let $(X,\mathcal{S})$ be a set system with $\VC(X,\mathcal{S})<\infty$. Is it the case that $(s_k)_{k\in\mathbb{N}}$ is non-decreasing? 
\end{question}

\section{Non-Standard Terminology and Notation}\label{sec:terminology-notation}

Having briefly discussed the history of VC-dimension of $k$-fold unions geometric families and introduced some of the abstract machinery that is used later on, we now turn our attention to the main problem we are concerned with. As mentioned in the previous section, we are interested in $k$-fold unions of lines in the plane, for small $k$. Let $\mathcal{F}_1$ denote the family of all lines in $\mathbb{R}^2$.
\[
    \mathcal{F}_1 := \left\{ \{(x,y):y=ax+b\right\} : a,b\in\R\}. 
\]
To fix notation, for the remainder of this paper, let $\mathcal{F}_k = (\mathcal{F}_1)^{\cup k}$. 

The arguments in this paper are essentially elementary and do not make use of any machinery heavier than pure incidence combinatorics. To make the exposition easier we introduce some (non-standard) terminology, which is all contained in this section. Thus, this section functions essentially as a small dictionary for the remainder of the paper, so the reader may skip this section in the first read, and refer back to it accordingly.

\begin{notation}
    Given two distinct points $a, b\in\R^2$, we denote by $l_{a,b}$ the unique line passing through them. Given a set $P\subseteq \R^2$ we write $\Scal_{P}^2 = \{l_{a,b}:a,b\in P, a\neq b\}$ for the set of all lines which go through (at least) two points in $P$. Clearly $|\Scal^2_{P}|\leq \binom{|P|}{2}$, and, in fact we have that $|\Scal^2_{P}| = \binom{|P|}{2}$ if and only if $P$ contains no three collinear points. 
    
    In general, for $n\geq 2$, we write $\Scal_{P}^n$ for the set of all lines $l\subseteq\R^2$ such that $|l\cap P|\geq n$. Let $\Scal_{P} = \{l\in\F_1 : l\cap P\neq \emptyset\}$ for the set of all lines in $\R^2$ that pass through some point in $P$.
\end{notation} 

\begin{definition} 
Let $P\subseteq \R^2$. We refer to elements of $P$ as \emph{$P$-points}. We say that a line $l$ is an \emph{$n$-line with respect to $P$} if $|l\cap P|=n$. When $P$ is clear from the context, we drop the reference to $P$ and speak simply of an \emph{$n$-line}. 
\end{definition}

\begin{definition}
    Let $P\subseteq\R^2$ be a set of points. Given $A\subseteq P$ and distinct $a,b\in A$ we say that $a,b$ \emph{pair inside $A$} if $l_{a,b}\cap P\subseteq A$.
\end{definition}

That is, $a,b\in A\subseteq P$ pair inside $A$ if $l_{a,b}\cap (P\setminus A) =\es$.

\begin{definition}
Let $P\subseteq\R^2$ be a set of points and $\Scal\subseteq\F_1$ a set of lines.
\begin{enumerate}
    \item A line $l\in\F_1$ is a \emph{cross-line with respect to $P$ and $\mathcal{S}$} if for each $l' \in S$ we have that $|l\cap l^\prime\cap P| = 1$. When it is clear from the context, we suppress mention of $P$ and $\mathcal{S}$.
    
    \item If $l\in\F_1$ is a cross-line and $y\in l$ we say $l$ is a \emph{$y$-cross-line}.
    
    \item A point $p\in P$ is an \emph{$n$-node} if $p$ belongs to exactly $n$ cross-lines.
    
    \item $\Scal$ \emph{covers} $P$ if $P\subseteq \bigcup_{l\in \Scal} l$.
\end{enumerate}
\end{definition}

\begin{example}
    
The following diagram depicts a $2$-node as witnessed by two $y$-cross-lines (with respect to the horizontal lines).

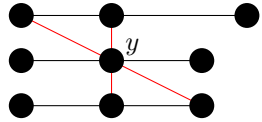
\begin{center}
\begin{tikzpicture}[scale=0.6]
\tikzset{dot/.style={fill=black,circle}}
{
\draw (0,0) -- (4,0);

\draw (0,1) -- (4,1);

\draw (0,2) -- (5,2);

\draw[red] (2,0) -- (2,2);

\draw[red] (4,0) -- (0,2);

}

\node[dot] at (0,0){};

\node[dot] at (2,0){};

\node[dot] at (4,0){};

\node[dot] at (0,1){};
\node[above, right] at (2.1,1.3) {$y$};
\node[dot] at (2,1){};

\node[dot] at (4,1){};

\node[dot] at (0,2){};

\node[dot] at (2,2){};

\node[dot] at (5,2){};

\end{tikzpicture} 
\end{center}

\end{example}

\begin{definition}
Let $P\subseteq\R^2$. We define $\collin(P)$ to be the maximum number of distinct collinear points contained in $P$.
\end{definition}

Observe that for any set $P\subseteq R$ of size $|P|\ge 2$ we have that $\collin(P) \geq 2$. 

\begin{definition}
Let $P\subseteq \R^2$ be a set of points and let $l$ be a line. We say that a line $l'$ is an \emph{ordinary $n$-line} with respect to $l$ if $|l'\cap l\cap P|=1$ and $|l'\cap P|=n$. 

Let $O_n(l,P)$ be the set of ordinary $n$-lines with respect to $l$ and $P$. We say that a line $l'$ is an \emph{ordinary $n$-line} if there is $l\in\mathcal{F}_1$ such that $l'\in O_n(l,P)$. In general, we do not specify $P$ when it is clear from context, and, in this case, we use the notation $O_n(l)$.

\end{definition}

We also write $O_{\geq n}(l,P)$ for the set $\bigcup_{i=n}^{\infty}O_i(l,P)$.

\begin{definition}
Let $A, B\subseteq\R^2$ be disjoint subsets of a set $P\subseteq \R^2$. 
\begin{itemize}
    \item A \emph{matching of $A$ and $B$ in $P$} is a set of lines in $\R^2$ whose union contains $A\cup B$, and such that each of the lines contains at most one element from each of $A$ and $B$, and no elements from $P\setminus (A\cup B)$.
    \item Let $a\in A$ and $b\in B$. We say that $(a,b)$ is a \textit{bad pair (with respect to $A$ and $B$, in $P$)} if $l_{a,b}\cap (P\setminus (A\cup B)) \neq \emptyset$, and a \emph{good pair} otherwise.
    \item Given distinct $a, a'\in A$ and $b, b'\in B$, we call $(a, a', b, b')$ a \emph{bad quadruple} if all of $(a, b)$, $(a, b')$, $(a', b)$, $(a', b')$ are bad pairs.
\end{itemize}
\end{definition}

The following lemma is technical but quite indicative of how our arguments are structured.

\begin{lemma}\label{3smatch-withconds}
Let $A=\{a_1, a_2, a_3\}$, $B=\{b_4, b_5, b_6\}$, and $P=A\cup B\cup \{q_1, q_2\}\subseteq \mathbb{R}^2$. Assume that: (1) $a_1,a_2,a_3$ are collinear; (2) $\collin(P)<4$; and (3) there are no bad quadruples in $P$. Then there is a matching of $A$ and $B$ in $P$. 
\end{lemma} 

\begin{proof} 
Observe that, since $\collin(P)<4$ and $|P\setminus (A\cup B)|=2$ for every $a\in A$ (respectively, $b\in B$) there are at most two $b\in B$ (respectively, $a\in A$) such that $(a,b)$ is a bad pair. Of course, if each element of $A$ is in at most one bad pair, then it is easy to check that the result follows, because, by the previous observation, we must be able to find a good pair for each $b\in B$. So, we have to consider what happens when elements of $A$ are in more than one bad pair. 

If there are two elements in $A$, say $a_1, a_2$, in two bad pairs each, then there is $b\in B$, such that $(a_1, b), (a_2, b)$ are both bad pairs. But then the previous observation, $(a_3, b)$ is not a bad pair and each of $a_1$ and $a_2$ must be is in a good pair, say $(a_1, b')$ and $(a_2, b'')$, where $b', b''\ne b$. By assumption (3), we must also have $b'\neq b''$ (otherwise, either $(a_1, a_2, b, b')$ or  $(a_1, a_2, b, b'')$ would be a bad quadruple). Hence the lines $l_{a_3, b}$, $l_{a_1, b'}$ and $l_{a_2, b''}$ provide a desired matching. 

Finally, if $a_1\in A$ is the only element of $A$ that is in two bad pairs, say $(a_1, b_4)$ and $(a_1, b_5)$, then the remaining possible bad pairs can be of the form $(a_2, b)$ and $(a_3, b')$, where $b, b'$ cannot both be $b_4$ and cannot both be $b_5$. It follows that for some $b, b'\in \{b_4, b_5\}$, we have that $(a_2, b_4)$ and $(a_3, b_5)$ are both good pairs. Hence the lines $l_{a_1,b_6}$, $l_{a_2, b_4}$ and $l_{a_3,b_5}$ provide a desired matching.
\end{proof}

We also record here the following general lemma which is useful in \Cref{sec:vc-3-lines}.

\begin{lemma} \label{lemma for line uniqueness}
    Fix $m,n\in\Nbb_{\geq1}$. Let $\{L_1,...,L_m\}$ be a set of $m$ distinct lines and $\{x_1,...,x_n\} \sub \cup_{i=1}^m L_i$ a set of $n$ distinct collinear points. If $m<n$, then there is $1\leq i\leq m$ such that $L_i$ is the unique line passing through each point in the set $\{x_1,...,x_n\}$.
\end{lemma} 

\begin{proof}
    Since $n>m$, there is some $i$ with $|L_i \cap \{x_1,...,x_n\}|\geq 2$. This line passes through the entire set since it is a set of collinear points.
\end{proof} 

We end this section with another easy lemma that we make use of often and possibly without reference in the sequel.

\begin{lemma}\label{basic-fact}
Let $n\in\mathbb{N}$.
\begin{enumerate}
    \item \label{no-n+1-collin} No $n+2$ collinear points may be shattered by $\mathcal{F}_n$. More generally, if $\collin(P)\geq n+2$ then $P$ cannot be shattered by $\F_n$.
    \item \label{n-lines} If $P\subseteq\R^2$ is shattered by $\F_n$, then there exist $n$ lines $L_1,\dots,L_n$ such that $P\subseteq \bigcup_{i\leq n}L_i$.
    \item \label{no-2-n+1-disjoint} No two disjoint sets $A$ and $B$,
    each consisting of $n+1$ collinear points,
    can be shattered by $\mathcal{F}_n$. 

\end{enumerate}
\end{lemma}
\begin{proof}
For (\ref{no-n+1-collin}), any subset  of size $n+1$  of the $n+2$ collinear points cannot be isolated by $\CF_n$. For (\ref{n-lines}), unless such lines exist, the whole set $P$ cannot be isolated. For (\ref{no-2-n+1-disjoint}), in order to be able to isolate a set consisting of $n$ points from $A$ together with all of $B$, the points of $B$ must be collinear with one point in $A$. Then apply (\ref{no-n+1-collin}).
\end{proof}

\section{The VC-dimension of \texorpdfstring{$\F_2$}{F2}}\label{sec:vc-2-lines}
As an introductory result, we compute $\VC(\mathcal{F}_2) = 5$, and we identify the necessary and sufficient conditions for 
a set $P\subseteq\R^2$ to be shattered by $\F_2$.

We start with the following easy lemma, which intuitively says that given any set of at most four points $P$, such that $\collin(P)<3$, and any two points outside $P$ we can shatter $P$ using lines which avoid these two points. 

\begin{lemma}\label{4avoid2}
Let $P$ be any set of at most four points in $\R^2$, such that no three points in $P$ are collinear, and let $q_1, q_2\in\R^2\setminus P$. Then the family
\[
    \F_2^\prime = \{(l_1\cup l_2)\in\F_2:(l_1\cup l_2)\cap \{q_1, q_2\} = \emptyset\}
\]
shatters $P$. 
\end{lemma}

\begin{proof}
This is a straightforward check. 
\end{proof}

\begin{proposition}\label{prop:F_2-shattering}
Let $P\subseteq \R^2$ be a set of $5$ points. Then $P$ is shattered by $\F_2$ if and only if both of the following conditions are satisfied. 
\begin{enumerate}
    \item \label{F_2-shat: cover}  $P$ is covered by two lines, that is, there exist $L_1,L_2\in\F_1$ such that $P = (L_1\cup L_2)\cap P$.
    \item \label{F_2-shat: collin} $P$ does not contain four collinear points.
\end{enumerate}
\end{proposition}

\begin{proof} 
First, suppose $P$ is shattered by $\F_2$. We must have (\ref{F_2-shat: collin}), by \Cref{basic-fact}(\ref{no-n+1-collin}), and since $P$ is contained in the union of $2$ lines, it is clear that (\ref{F_2-shat: cover}) must hold.

For the other direction, observe that condition (\ref{F_2-shat: cover}) implies that $P$ contains three collinear points. Suppose that $P\subseteq\R^2$ satisfies conditions (\ref{F_2-shat: cover}) and (\ref{F_2-shat: collin}). Given any $A\subseteq P$, we must show that $\F_2$ isolates $A$. Of course, if $|A|\leq 2$ then this is immediate, and, if $A=P$, then condition (\ref{F_2-shat: cover}) ensures that $\F_2$ isolates $A$. So, suppose that $|A|=4$, then either $A$ contains $3$ collinear points, in which case we are done easily by condition (\ref{F_2-shat: cover}), or $A$ contains no sets of three collinear points, in which case, we are done by \Cref{4avoid2}. If $|A|=3$, then either the points in $A$ are collinear, in which case $\F_2$ isolates $A$, or the points in $A$ are not all collinear, in which case the result follows from \Cref{4avoid2} applied to $A$ with $\{q_1,q_2\}=P\setminus A$).
\end{proof}

\begin{corollary}\label{vc-2-lines}
    $\VC(\mathcal{F}_2)=5$.
\end{corollary}

\begin{proof}
First, it is easy to see that there are sets of $5$ points satisfying the conditions of \Cref{prop:F_2-shattering}, and hence $\VC(\F)\geq 5$. For the other inequality, suppose that $\F_2$ shatters a set $P$ of size 6. Then, by \Cref{basic-fact}(\ref{no-n+1-collin}) $P$ must consist of two disjoint sets of three collinear points, but then by \Cref{basic-fact}(\ref{no-2-n+1-disjoint}) it is clear that $\F_2$ cannot shatter $P$, hence $\VC(\F_2)\leq 5$, and the result follows.
\end{proof}

\section{The VC-dimension of \texorpdfstring{$\F_3$}{F3}}\label{sec:vc-3-lines}

We now turn our attention to $\F_3$. We start with the following easy proposition.

\begin{proposition}
    $\VC(\F_3)\leq 9$.
\end{proposition}

\begin{proof}
Suppose for a contradiction that there is some $P\subseteq\R^2$ with $|P|=10$, which is shattered by $\F_3$. By \Cref{basic-fact}(\ref{no-n+1-collin}), $P$ cannot contain five collinear points and, thus must be contained entirely on three distinct lines. This implies, in particular, that there is a line $l$ with $|l\cap P|=4$. But then each element of $\{A \cup (l\cap P):A\subseteq (P\setminus l)\}$ is isolated by a union of three lines, and clearly, $l$ must be one of them. This implies that $P\setminus l$ is shattered by two lines, contradicting \Cref{vc-2-lines}.
\end{proof}

Hence, we may now focus on sets $P\subseteq\R^2$ such that $|P| \leq 9$. The aim of this section is to prove that such a $P\subseteq\R^2$ is shattered by $\F_3$ if and only if $P$ satisfies 

\begin{itemize}
    \item[(\namedlabel{cond:0}{O})]  $P$ is covered by at most three lines.
\end{itemize}  
and one of two collections of additional axioms listed under Case A and Case B in the text below.

For the remainder of this section, we always assume that $P\subseteq \Rbb^2$ consists of at most $9$ points and satisfies (\ref{cond:0}). We proceed by dedicating a subsection to each of the additional sets of axioms. 

Let $m_P:= \min \{|C|\in \mathbb{N}: C\subseteq \F_1, C$ covers $P\}$. In this notation, Condition (\ref{cond:0}) says $m_P\leq 3$. Condition (\ref{cond:0}) is necessary for any set $P\subseteq\mathbb{R}^2$ to be shattered by $\mathcal{F}_3$, so we assume it throughout.

\subsection*{Case A: There are 4 collinear points in \texorpdfstring{$P$}{P}} Here, we  prove that a set $P\sub \R^2$ of $9$ points satisfying condition (\ref{cond:0}) and containing $4$ collinear points is shattered by $\F_3$ if and only if the following two conditions hold. 

\begin{itemize}
    \item[(\namedlabel{cond:b}{A1})] For any $A\subseteq P$ with $|A|=4$ there are distinct $a,a'\in A$ that pair inside $A$.
    \item[(\namedlabel{cond:c}{A2})] For any line $l$ with $|l\cap P|\geq 4$ and $x\in l\cap P$, there is at least one other line $l'$ containing $x$ such that $|l'\cap P|\geq 3$.
\end{itemize}

\begin{example}\label{ex:set1}
    The following set of nine points in $\R^2$ contains $4$ collinear points and satisfies $\{$(\ref{cond:0}),(\ref{cond:b}),(\ref{cond:c})$\}$.
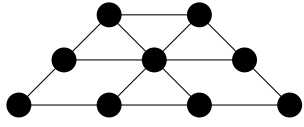
    \begin{center}
    \begin{tikzpicture}[scale=0.6]
\tikzset{dot/.style={fill=black,circle}}
{
\draw (3,3) -- (7,3);
\node at (1.6,3) {};
\draw (2,2) -- (8,2);
\node at (0.6,2) {};
\draw (2,2) -- (4,4);
\node at (2,1.4) {};
\draw (4,2) -- (6,4);
\node at (4,1.4) {};
\draw (6,2) -- (4,4);
\node at (6,1.4) {};
\draw (8,2) -- (6,4);
\node at (8,1.4) {};
\draw (4,4) -- (6,4);
}

\node[above, right] at (2.3,2.2) {};
\node[dot] at (2,2){};
\node[above, right] at (4.3,2.2) {};
\node[dot] at (4,2){};
\node[above, right] at (6.3,2.2) {};
\node[dot] at (6,2){};
\node[above, right] at (8.3,2.2) {};
\node[dot] at (8,2){};
\node[above, right] at (3.3,3.2) {};
\node[dot] at (3,3){};
\node[above, right] at (5.3,3.2) {};
\node[dot] at (5,3){};
\node[above, right] at (7.3,3.2) {};
\node[dot] at (7,3){};
\node[above, right] at (4.3,4.2) {};
\node[dot] at (4,4){};
\node[above, right] at (6.3,4.2) {};
\node[dot] at (6,4){};

\end{tikzpicture}      
\end{center}

\end{example}

For the remainder of Case A, we fix (unless otherwise stated) a set $P$ of nine points, containing $4$ collinear points and satisfying (\ref{cond:0}), (\ref{cond:b}) and (\ref{cond:c}). The following lemma is an easy consequence of our assumptions.

\begin{lemma}\label{cond:d} 
    For any $x\in P$ and any line $l$ such that $|l\cap P| = 4$ (which exists since we are in Case A), if $x\not\in l$, then there do not exist  three distinct lines $l_1,l_2,l_3$ such that $|l_i\cap P\cap l|=1$, $|l_i\cap P|\geq 3$ and $x\in l_i$.
\end{lemma}

\begin{lemma}\label{conditions-imply-no5collin}
$\collin(P) = 4$.
\end{lemma}

\begin{proof}
Suppose $\collin(P)\geq 5$ and let $\{a_1,\dots,a_5\}\subseteq P$ be a set of $5$ collinear points. By Condition (\ref{cond:b}), every set of four points in $P$ contains two points that pair inside it. In particular, $\{a_1,\dots,a_4\}$ contains two points that pair inside $\{a_1,\dots,a_4\}$, but that is impossible, since, for any two distinct $i,j\leq 4$ we have that $l_{a_i,a_j}$ contains $a_5$.
\end{proof}

\begin{lemma}\label{conditions-imply-no2disjoint4collin}
Any two sets of four collinear points in $P$ intersect.
\end{lemma} 

\begin{proof}
Suppose otherwise. Let $\{a_1,a_2,a_3,a_4\}$ lying on a line $l_a$ and $\{b_1,b_2,b_3,b_4\}$ lying on a line $l_b$ be disjoint subsets of $P$. By (\ref{cond:c}), there are lines $k_1,k_2,k_3\in O_{\geq3}(l_a)$ such that $a_i \in k_i$ for each $i\in \{1,2,3\}$. By \Cref{conditions-imply-no5collin}, $k_i\neq l_b$ for each $i\in \{1,2,3\}$. Since $|k_i\cap P|\geq 3$ for each $i\in \{1,2,3\}$, we have that $(P\setminus (l_a\cup l_b))\cap k_i \neq \es$. Since $|P\setminus (l_a\cup l_b)|=1$, we may let $P\setminus (l_a\cup l_b)=\{c\}$, and, so, $c\in k_i$ for each $i\in \{1,2,3\}$. But then this contradicts \Cref{cond:d}.
\end{proof}

\begin{lemma}\label{conditions-imply-unique-Li}
There is a unique set $S\in\F_3$ which covers $P$.
\end{lemma}

\begin{proof}
By assumption, $P$ contains a set of four collinear points. By \Cref{lemma for line uniqueness} if $S=L_1\cup L_2\cup L_3\in\F_3$ covers $P$ it must contain a line passing through these four points, say $L_1$. By \Cref{conditions-imply-no5collin}, $|P\sm L_1|=5$. $P\sm L_1 \sub L_2 \cup L_3$, so $P\sm L_1$ contains a set of three collinear points. By \Cref{lemma for line uniqueness} again, one of $L_2,L_3$ passes through these points, say this is $L_2$. By \Cref{conditions-imply-no2disjoint4collin}, $|P\sm (L_1\cup L_2)|=2$. Thus, $L_3$ must be the unique line passing through both points in this set. This process identifies $S$ and thus $L_1,L_2,L_3$ uniquely up to renumbering.
\end{proof}

In virtue of \Cref{conditions-imply-unique-Li}, we fix an enumeration of the unique three lines $L_1,L_2,L_3$ covering $P$.

\begin{restatable}{lemma}{prelim}\label{lem:set1-preparatory} \,
\begin{enumerate}
    \item \label{O3line-intersect-3collin} Let $l$ be a $4$-line and $k\in O_3(l)$. Then $k$ intersects all sets of three collinear points in $P\setminus l$. In particular, for all $A\subseteq P$, if $\collin(A)\geq 4$, then $k\cap A\neq \emptyset$.
    \item  \label{nobadquadruple} Given disjoint $A,B \subseteq P$ where $A$ is a set of four collinear points and $B$ a set of (any) three points, there is no bad quadruple with respect to these sets.
    \item \label{set1-4collin} Given $A\subseteq P$, if $\collin(A) = 4$ then $A$ can be isolated by $\F_3$.
    \item \label{no3collin-leq6} Given $A\subseteq P$ with $|A|\in \{5,6\}$ which contains no three collinear points, $A$ can be isolated by $\F_3$. 
    \item \label{3collin-leq6} If $A\subseteq P$ is such that:
    \begin{enumerate} 
        \item $A\subseteq P$ has $|A|=6$ (respectively, $|A|=5)$
        \item No ordinary 3-line in $P$ contains three points in $A$
        \item No four points in $A$ are collinear.
        \item There are three collinear points in $A$. Call those points $a_1, a_2, a_3$, the rest of the points in $A$ by $b_1, b_2, b_3$ (respectively, $b_1, b_2$), and let $P\setminus A=\{q_1, q_2, q_3\}$ (resp $P\setminus A =\{q_1, q_2, q_3,q_4\}$
  \end{enumerate}
    Then $A$ can be isolated by $\F_3$.
\end{enumerate}
\end{restatable}

\begin{proof}
This is a long and technical, yet elementary proof. The brave reader can find it in \Cref{pf:set1-preparatory}.
\end{proof}

\begin{proposition}[Sufficiency]\label{3linesleft} 
$P$ is shattered by $\F_3$.
\end{proposition}

\begin{proof} 
We must show that each $A\subseteq P$ can be isolated by $\F_3$. Let $A\subseteq P$. Without loss of generality, we may assume that $\collin(A)<4$, otherwise the result follows by \Cref{lem:set1-preparatory}(\ref{set1-4collin}). The rest of the argument is by cases on $|A|$, for $A\subseteq P$ containing no four collinear points. If $|A|\leq 3$ then the result follows trivially, if $|A| = 4$, then the result follows from (\ref{cond:b}), and if $|A| = 9$ then the result follows from (\ref{cond:0}). For the rest of the proof, fix a line $l$ passing through four points $\{a_1,a_2,a_3,a_4\}$ and let $P\setminus l= \{b_1,b_2,b_3,b_4,b_5\}$ where $\{b_1,b_2,b_3\}$ are collinear.

Suppose $|A|=8$. Since $A$ contains no set of four collinear points, $\{a_1,a_2,a_3,a_4\}\nsubseteq A$. Without loss of generality, we may assume that $a_4\notin A$. By \Cref{cond:c}, there are distinct ordinary $3$-lines $k_1,k_2,k_3$ with $a_i\in k_i$, for $i\in[3]$. Since the $k_i$ are ordinary $3$-lines they must intersect $\{b_1,b_2,b_3\}$ at most once. So, each $k_i$ must intersect $\{b_4,b_5\}$ non-trivially. Without loss of generality, we assume that $b_4\in k_1\cap k_2$. By \Cref{cond:d}, this implies that $b_5\in k_3$. Since $\{b_1,b_2,b_3\} \subseteq P\setminus l$ is a set of three collinear points, \Cref{lem:set1-preparatory}(\ref{O3line-intersect-3collin}) implies that each $k_i$ must intersect $\{b_1,b_2,b_3\}$ at least once. Without loss of generality, we assume that $b_3\in K_3$. Since $k_1\neq k_2$, we must have that $b_3\notin k_1\cap k_2$. So, we may assume without loss of generality $b_3\notin k_1$ and that $k_1\cap P= \{a_1, b_1, b_4\}$ which is entirely disjoint from $k_3\cap P$. Hence, lines $k_1$, $k_3$ and $l_{a_2, b_2}$ isolate $A$ (since none of these lines pass through $a_4$).

For the remainder of the proof, suppose that $|A|\in\{5,6,7\}$ and, as always, $\collin(A)<4$.

\textit{Claim 1.} If there is an ordinary $3$-line $k$ such that $P\cap k \subseteq A$, then $\F_3$ shatters $A$.

\textit{Proof of Claim 1.} If $\collin(A\setminus k) = 3$, fix a line $k'$ passing through $3$ points in $A\sm k$. If $|k'\cap P|>3$, $|k'\cap P|=4$. By \Cref{lem:set1-preparatory}(\ref{O3line-intersect-3collin}), in this case, $k'\cap P \sub A$. If $|k'\cap P|=3$, it is trivial that $k'\cap P \sub A$. So, in either case, a subset $A'\sub A$ of size $6$ is isolated by two lines $k,k'$. It follows immediately from this that $A$ may be isolated by $\F_3$ since $|A|\leq 7$. 

Hence, we may assume that $\collin(A\setminus k) = 2$. Note that $|A\setminus k|\leq 4$ and $|P\setminus k|=6$. If $|A\setminus k|=4$, then the conditions of \Cref{4avoid2} must apply, and hence in this case we are done. If $|A\setminus k|\leq 2$, then we are done trivially. So, we may assume that we are precisely in a situation where $|A\setminus k|=3$. For any $c\in P\setminus A$, we may assume that $\collin((A\setminus K)\cup \{c\}) = 3$, for if $A\setminus k$ is contained in some $A'\subseteq P$ satisfying the conditions of \Cref{4avoid2}, the result again follows.

By \Cref{lem:set1-preparatory}(\ref{O3line-intersect-3collin}) $\collin(P\setminus k)\leq 3$, so we must have that the line through each pair of points in $A\setminus k$ passes through exactly one point in $P\setminus A$. Since these lines are distinct, we may label them as follows: $k\cap P:= \{x_1,x_2,x_3\}$, $A\setminus k:= \{y_1,y_2,y_3\}$ and $P\setminus A:= \{z_{12},z_{13},z_{23}\}$ where $z_{ij}$ lies on $l_{y_i,y_j}$. By condition (\ref{cond:b}) applied to $\{x_i,y_1,y_2,y_3\}$, for each $i$ there is $j$ such that $x_i, y_j$ pair inside $A$. 

First, we show that each $y_j$ pairs inside $A$ with some $x_i$. Suppose for contradiction that $y_1$ does not pair with any $x_i$. Then for each $i$, $l_{y_1,x_i}\cap \{z_{12},z_{13},z_{23}\} \neq \es$. Now, $k$ is an ordinary $3$-line, so for any $p\in P$ $\collin(\{x_1,x_2,x_3,p\}) = 3$. So, possibly after renumbering, $\{x_1,y_1,z_{12}\}$, $\{x_2, y_1, z_{13}\}$, $\{x_3, y_1,z_{23}\}$ are sets of collinear points. This implies that $\{x_1,y_1,z_{12}, y_2\}$, $\{x_2,y_1,z_{13}, y_3\}$ are sets of collinear points. But then $\{x_3,y_1,z_{23}\}$ lie on an ordinary 3-line $k'$ with respect to $l'$, where $l'$ is the line passing through $\{x_1,y_1,z_{12}, y_2\}$. This is an ordinary 3-line that does not intersect $\{x_2,z_{13}, y_3\}$, a set of three collinear points in $P\setminus l'$, which contradicts \Cref{lem:set1-preparatory}(\ref{O3line-intersect-3collin}). Of course, we may repeat the same argument with any $y_j$ in place of $y_1$. It follows that each $y_j$ pairs inside $A$ with some $x_i$. 

Note that if $y_1, y_2$ pair inside $A$ only with say $x_1$, then the set $\{x_2,x_3,y_1,y_2\}$ does not satisfy condition (\ref{cond:b}), so, possibly after renumbering, we have that $x_1, y_1$ and $x_2,y_2$ pair inside $A$. If $x_3,y_3$ pair inside $A$, then we're done, so suppose not. Then $y_3$ must pair inside $A$ with $x_1$ or $x_2$ and $x_3$ must pair inside $A$ with $y_1$ or $y_2$. If $y_3$ pairs inside $A$ with both $x_1$ and $x_2$, then we are done, so suppose that $y_3$ pairs inside $A$ with $x_1$ and not $x_2$. If $x_3$ pairs with $y_1$ we are done, so suppose otherwise and assume $x_3$ pairs inside $A$ with $y_2$. Now if $x_2$ does not pair with $y_1$, the set $\{y_1,y_3,x_2,x_3\}$ will contradict condition (\ref{cond:b}), so we must have that $x_2, y_1$ pair inside $A$, and we are done. \hfill$\blacksquare_{\textit{Claim }1}$

\textit{Claim 2.} If there is no ordinary $3$-line $k$ such that $P\cap k \subseteq A$, then $|A|\leq 6$.

\textit{Proof of Claim 2.} Observe that there is some $x\in l\cap P$ which is not in $A$.

By (\ref{cond:c}), there are at least three distinct lines $k_1,k_2,k_3\in O_3(l)\cup O_4(l)$ intersecting $l$ at each $y\in (l\cap P)\setminus \{x\}$. By assumption, for each $i$, $k_i\cap P$ is not contained in $A$. By \Cref{cond:d}, each $z\in P\setminus l$ is contained in at most two of the $k_i$. So, $k_1 \cap k_2 \cap k_3 \cap P =\es$. Thus, in order that each $k_i\cap P$ is not contained in $A$, there must be two further points in addition to $x$ that are in $P\sm A$. 

\hfill$\blacksquare_{\textit{Claim }2}$

If $A$ satisfies the assumption of Claim $2$, then $|A|\in\{5,6\}$, and such sets can be shattered by $\F_3$, by \Cref{lem:set1-preparatory}(\ref{no3collin-leq6}) and \Cref{lem:set1-preparatory}(\ref{3collin-leq6}).
\end{proof}

\begin{proposition}[Necessity]\label{necessity-A}
If $P\subseteq \mathbb{R}^2$ with $|P|=9$ is shattered by $\F_3$ and contains four collinear points then $P$ must satisfy (\ref{cond:0}), (\ref{cond:b}), and (\ref{cond:c}).
\end{proposition}

\begin{proof}
Let $P\subseteq \mathbb{R}^2$ with $|P|=9$, and assume that $P$ is shattered by $\mathcal{F}_3$. Assume further that $P$ contains a set of $4$ collinear points. We note that (\ref{cond:b}) must hold, for otherwise, there would be some $A\subseteq P$ of size $4$ which is not isolated by $\F_3$. We also observe that (\ref{cond:0}) holds, for otherwise $P$ would not be isolated. 

It only remains to show that (\ref{cond:c}) holds. To this end, let $l$ be a line such that $|P\cap l| = 4$. Let $l\cap P = \{a_1,\dots a_4\}$ and $P\setminus l = \{b_1,..., b_5\}$. Assume towards a contradiction, that every line $k\neq l$ with $a_1 \in k$ is such that $|k \cap P| \leq 2$.  There are two cases to consider. 

\textbf{Case I:} Suppose that there is a unique set of four collinear points in $P$. Thus, the set $\{a_1,\dots,a_4\}$ is the only set of four collinear points in $P$. We know, from (\ref{cond:0}) and \Cref{basic-fact}(\ref{no-2-n+1-disjoint}), that $\collin(\{b_1,\dots,b_5\}) = 3$. So, without loss of generality, we may assume that $\{b_1, b_2, b_3\}$ are collinear. 

For each $i\in[4]$, let $A_i = P\setminus\{a_i\}$ and note that, by assumption, $A_i$ is isolated by $\F_3$. But, by assumption, every line $k\neq l$ containing $a_1$ has $|P\cap k| \leq 2$. Since we can isolate $A_4$ there must be lines $k_2, k_3$ with $a_i \in k_i$, $|k_i\cap P|=3$ and $(k_2\cap k_3)\cap P=\emptyset$. This forces, without loss of generality, $\collin(\{a_2, b_1, b_4\}) = 3$ and $\collin(\{a_3, b_2, b_5\}) = 3$ (here, we are picking an element from $\{b_1,b_2,b_3\}$ and an element from $\{b_4,b_5\}$ for each line). 

We now turn our attention to $A_3$ and $A_2$, to derive a contradiction. We will use repeatedly that, since $l$ is the only set of collinear points with size $4$, to isolate $A_i$ there must be two disjoint ordinary $3$-lines (with respect to $l$) contained in $A_i$. Assume first, that the only ordinary $3$-line passing through $a_4$ passes through $b_3$, so without loss of generality suppose that $\{a_4,b_3,b_4\}$ lie on an ordinary $3$-line. Then arguing with $A_3$ we see that there must be another ordinary $3$-line passing through $a_2$ which must pass through $b_2,b_5$, which is a contradiction, since this would witness another set of four collinear points. Hence, there must be an ordinary $3$-line passing through $a_4$ and not $b_3$. Without loss of generality, suppose that $\{a_4, b_2, b_4\}$ are collinear. Now, there can be no more ordinary $3$-lines containing $b_4$ by \Cref{cond:d}, so there is a unique ordinary $3$-line containing $a_3$. That is, the line through $\{a_3,b_2,b_5\}$. We see this since any ordinary $3$-line passing through $a_3$ must pass through $b_4$ or $b_5$, and by the observation above it cannot pass through $b_4$. But then, we cannot isolate $A_2$ since an ordinary $3$-line must pass through $b_4$, and hence $b_2$, or $b_5$. In either case, it intersects the unique ordinary $3$-line containing $a_3$.

\textbf{Case II:} Suppose there is more than one set of four collinear points in $P$. By \Cref{basic-fact}(\ref{no-2-n+1-disjoint}), in order for $\F_3$ to shatter $P$, any two sets of four collinear points must intersect. So, without loss of generality suppose both $\{a_1,...,a_4\}$ and $\{a_4, b_1, b_2, b_3\}$ are sets of collinear points. Let $k_{i,j}$ denote the line through $a_i$ and $b_j$. We know from \Cref{cond:d} that at most one of $k_{2,1},k_{2,2},k_{2,3}$ intersects $\{b_4,b_5\}$, so we may assume that $(k_{2,1}\cup k_{2,2})\cap \{b_4,b_5\}=\emptyset$. The same must be true for $k_{1,2}$ and $k_{1,1}$, by the assumption on $a_1$. Let $B = \{a_1,a_2,b_1,b_2,a_4\}$. Since $\F_3$ shatters $P$, both $B\cup \{b_4, b_5\}$ and $B\cup \{b_4\}$ can be isolated by $\F_3$. Suppose $B\cup \{b_4,b_5\}$ is isolated by $\F_3$. Then, two of the three lines that isolate $B\cup \{b_4,b_5\}$ must be either: $k_{1,1}$ and $k_{2,2}$, or $k_{1,2}$ and $k_{2,1}$. Therefore, $\collin(\{a_4,b_4,b_5\})=3$. But then, by similar reasoning, we cannot isolate $B\cup \{b_4\}$.

In either case, we have derived a contradiction, so (\ref{cond:c}) must hold, and this concludes the proof.
\end{proof}

\subsection*{Case B: \texorpdfstring{$P$}{P} does not contain 4 collinear points}
In this case, a set $P\sub \R^2$ of $9$ points satisfying condition (\ref{cond:0}) is shattered by $\F_3$ if and only if the following two conditions hold.
\begin{itemize}
    \item[(\namedlabel{cond:B}{B1})] For any choice of lines $\{L_1,\dots,L_{m_P}\}$ covering $P$ and for all $x\in P$, $x$ is contained in exactly two cross-lines with respect to $P$ and $\{L_{i_1},\dots,L_{i_{m_p}}:i_j\in[m_P],\text{ pairwise different}\}$. 
    \item[(\namedlabel{cond:C}{B2})] For any choice of lines $\{L_1, \dots ,L_{m_P}\}$ covering $P$, there exists a $y\in P$, two $y$-cross-lines $l_y,l_y^\prime$, and a cross-line $l$ with $y\notin l$ such that $l\cap l_y\neq \emptyset$, and $l\cap l_y' \neq \emptyset$, where these cross-lines are so with respect to $P$ and $\{L_1,\dots,L_{m_P}\}$.
\end{itemize}

\begin{example}\label{ex:set2}
    The following set of nine points in $\R^2$ does not contain four collinear points and satisfies $\{$(\ref{cond:0}),(\ref{cond:B}),(\ref{cond:C})$\}$.
    \begin{center}
\begin{tikzpicture}[scale=0.3]
\tikzset{dot/.style={fill=black,circle}}
{
\draw (0,0) -- (6,0);
\draw (1.5,3) -- (6,3);
\draw (3,0) -- (3,6);
\draw (0,0) -- (3,6);
\draw (0,0) -- (6,3);
\draw (3,3) -- (6,0);
\draw (6,3) -- (6,-6);
\draw (6,-6) -- (1.5,3);
\draw (6,-6) -- (3,6);
}
\node[above, right] at (0.2,0.2) {};
\node[dot] at (0,0){};
\node[above, right] at (3.2,0.2) {};
\node[dot] at (3,0){};
\node[above, right] at (6.2,0.2) {};
\node[dot] at (6,0){};
\node[above, right] at (1.7,3.2) {};
\node[dot] at (1.5,3){};
\node[above, right] at (3.2,3.2) {};
\node[dot] at (3,3){};
\node[above, right] at (6.2,3.2) {};
\node[dot] at (6,3){};
\node[above, right] at (3.2,6.2) {};
\node[dot] at (3,6){};
\node[above, right] at (6.2,-6.2){};
\node[dot] at (6,-6){};
\node[above, right] at (4.2,2.2) {};
\node[dot] at (4,2){};
\end{tikzpicture}
\end{center}
\end{example}

For the remainder of Case B, we fix (unless otherwise stated) a set $P$ of nine points, not containing any $4$ collinear points and satisfying (\ref{cond:0}), (\ref{cond:B}) and (\ref{cond:C}). The structure of this subsection closely follows that of Case A.

\begin{lemma}\label{Conditions-imply-unique-Li}
There is a unique set $S\in\F_3$ which covers $P$.
\end{lemma}

\begin{proof}
Fix a choice of lines $L_1,L_2,L_3$ such that $(L_1\cup L_2\cup L_3)\cap P=P$. Since no four points are collinear, $|L_i\cap P|=3$ for each $i\in \{1,2,3\}$. So, the sets $L_i\cap P$ are pairwise disjoint. Let $L_1\cap P=\{a_1,a_2,a_3\}$, $L_2\cap P=\{b_1,b_2,b_3\}$ and $L_3\cap P=\{c_1,c_2,c_3\}$ (so that $P=\{a_1,a_2,a_3,b_1,b_2,b_3,c_1,c_2,c_3\}$). 

Suppose for contradiction that there is another choice $L_1',L_2',L_3'$ such that $\{L_1,L_2,L_3\}\neq \{L_1',L_2',L_3'\}$. Note that if $L_i=L_j'$ for any $i,j\in \{1,2,3\}$, then $\{L_1,L_2,L_3\}= \{L_1',L_2',L_3'\}$. This follows by two applications of \Cref{lemma for line uniqueness}. So, we may assume that $\{L_1,L_2,L_3\}\cap \{L_1',L_2',L_3'\}= \es$. Without loss of generality, we may suppose that $L_i'\cap P = \{a_i,b_i,c_i\}$ for each $i\in \{1,2,3\}$.

Consider $y\in P$ as mentioned in (\ref{cond:C}). Without loss of generality, we may suppose that $y=a_1$, since the situation is symmetrical. By (\ref{cond:B}), $a_1$ is contained in exactly two cross-lines with respect to $\{L_1,L_2,L_3\}$. $L_1'$ is one such cross-line. We may assume without loss of generality that another one, $k$, passes through $\{a_1,b_2,c_3\}$ (or symmetrically $\{a_1,b_3,c_2\}$). By (\ref{cond:C}), there is a cross-line, $l$, with $l\cap L_1'\neq \es$, $l\cap k\neq \es$ and $a_1\notin l$. This implies that $l\neq k$. Since $l$ intersects $L_1'$ non-trivially, $l\notin \{L_2',L_3'\}$. Since $l\cap k\neq \es$, either $b_2\in l$ or $c_3\in l$. In either case, this point would be an element of three distinct cross-lines, contradicting (\ref{cond:B}).
\end{proof}

In virtue of \Cref{Conditions-imply-unique-Li}, we fix an enumeration of the unique three lines $L_1,L_2,L_3$ covering $P$.

\begin{definition}
Given $A\subseteq P$ with $|A|=4$, we say that $A$ is a \emph{pairing four} if some two elements of $A$ pair inside $A$.
\end{definition}

\begin{remark}
    Let $P$ be a set of nine points and $A\subseteq P$ with $|A| = 4$. Then $\F_3$ isolates $A$ if and only if $A$ is not a pairing four.
\end{remark}

In fact, we have the following lemma.
\begin{lemma}\label{no-non-pairing-fours}
If $A\subseteq P$ has $|A| = 4$, then $A$ is a pairing four.
\end{lemma}

\begin{proof}
To fix notation, let $P = \{a_1,a_2,a_3,b_1,b_2,b_3,c_1,c_2,c_3\}$ where $L_1\cap P= \{a_1,a_2,a_3\}$, $L_2\cap P=\{b_1,b_2,b_3\}$ and $L_3\cap P=\{c_1,c_2,c_3\}$.

Suppose first that $|A|=4$. Suppose without loss of generality that $a_1\in A$. If there is $x\in A\setminus \{a_1\}$ that pairs with $a_1$ in $A$ then we are done. So, suppose not. Then for each $x\in A\setminus \{a_1\}$, there is a line $l$ with $x,a_1\in l_x$ and $(P\setminus A)\cap l \neq \emptyset$. Such an $l_x$ must be an $a_1$-cross-line or be $L_1$. Since any such line $l_x$ passes through no more than three points (including $a_1$), $|(A\setminus \{a_1\})= l_x|\leq 1$. By (\ref{cond:B}), there are exactly two $a_1$-cross-lines. Since no four points are collinear, $l_x \neq l_{x'}$ whenever $x\neq x'$. So, since $|A\sm \{a_1\}|=3$, the $a_1$-cross-lines and $L_1$ each contain exactly one point in $A\setminus \{a_1\}$.
Without loss of generality, we may assume that the $a_1$-cross-lines pass through $\{a_1,b_1,c_1\}$ and $\{a_1,b_2,c_2\}$. We consider two cases, the rest are symmetric.

\noindent\textbf{Case I:} $A$ contains two elements of $L_1$ and two elements of one of $L_2$ or $L_3$. Without loss of generality, we assume that $A=\{a_1,a_2,b_1,b_2\}$. Suppose for contradiction that both $b_1,b_2$ are in $a_2$-cross-lines. By (\ref{cond:B}) and the fact that each of $b_1,b_2$ would be contained in an $a1$- and $a_2$-cross-line (two cross-lines in total), any $a_3$-cross-line must pass through $b_3$. Thus, there is only one $a_3$-cross-line, contradicting (\ref{cond:B}). Thus, one of $b_1,b_2$ pairs with $a_2$ in $A$ and we may isolate $A$.

\noindent\textbf{Case II:} $A$ contains two elements from $L_1$ and one element from each of $L_2, L_3$. Without loss of generality, we assume $A=\{a_1,b_1,c_2,a_2\}$. If $c_2,b_1,a_2$ are collinear, we are done since no four points are collinear. So, any $b_1$-cross-line intersects $\{c_2,a_2\}$ at most once. Since $\{a_1,b_1,c_1\}$ is a $b_1$-cross-line, by (\ref{cond:B}), at most one of $a_2,c_2$ is in a $b_1$-cross-line. Thus, $b_1$ pairs with at least one of $a_2,c_2$. 
\end{proof}

\begin{definition}
    Let $P\subseteq\R^2$ be a set of nine points. We say that $P$ is an \emph{$\mathsf{X}$-configuration} if there is a labelling of $P$ by $\{a_1,a_2,a_3,b_1,b_2,b_3,c_1,c_2,c_3\}$ such that exactly the following non-trivial collinearities occur. $\{a_1,a_2,a_3\}$, $\{b_1,b_2,b_3\}$, $\{c_1,c_2,c_3\}$, $\{a_1,b_1,c_2\}$, $\{a_1,b_2,c_3\}$, $\{a_2,b_1,c_1\}$, $\{a_2,b_3,c_3\}$, $\{a_3,b_2,c_1\}$, and $\{a_3,b_3,c_2\}$.
\end{definition}

\begin{example}\label{X-configuration}
This diagram shows an $\mathsf{X}$-configuration.

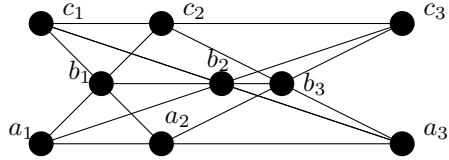
\begin{center}
\begin{tikzpicture}[scale=0.8]
\tikzset{dot/.style={fill=black,circle}}
{
\draw (0,0) -- (6,0);
\draw (1,1) -- (4,1);
\draw (0,2) -- (6,0);
\draw (0,0) -- (2,2);
\draw (0,0) -- (6,2);
\draw (2,0) -- (0,2);
\draw (2,0) -- (6,2);
\draw (6,0) -- (0,2);
\draw (6,0) -- (2,2);
\draw (0,2) -- (6,2);
}

\node[above, right] at (-0.7,0.2) {$a_1$};
\node[dot] at (0,0){};
\node[above, right] at (1.9,0.4) {$a_2$};
\node[dot] at (2,0){};
\node[above, right] at (6.2,0.2) {$a_3$};
\node[dot] at (6,0){};
\node[above, right] at (0.3,1.2) {$b_1$};
\node[dot] at (1,1){};
\node[above, right] at (2.6,1.4) {$b_2$};
\node[dot] at (3,1){};
\node[above, right] at (4.2,1) {$b_3$};
\node[dot] at (4,1){};
\node[above, right] at (0.2,2.2) {$c_1$};
\node[dot] at (0,2){};
\node[above, right] at (2.2,2.2) {$c_2$};
\node[dot] at (2,2){};
\node[above, right] at (6.2,2.2) {$c_3$};
\node[dot] at (6,2){};

\end{tikzpicture}      
\end{center}
    
\end{example}

\begin{remark}
    If $P$ is an $\mathsf{X}$-configuration, then it cannot satisfy condition (\ref{cond:C}).
\end{remark}

\begin{proposition}[Sufficiency] \label{3linesleft'}
$P$ is shattered by $\F_3$.
\end{proposition}

\begin{proof}
As in \Cref{no-non-pairing-fours}, we fix notation. Let $P = \{a_1,a_2,a_3,b_1,b_2,b_3,c_1,c_2,c_3\}$ where $L_1\cap P= \{a_1,a_2,a_3\}$, $L_2\cap P=\{b_1,b_2,b_3\}$ and $L_3\cap P=\{c_1,c_2,c_3\}$.

We must show that each $A\subseteq P$ can be isolated by $\F_3$. Let $A\subseteq P$. Trivially, if $|A|\leq 3$, then the result follows. By (\ref{cond:0}), the same is true if $|A|=9$. By \Cref{no-non-pairing-fours}, $A$ may be isolated if $|A|=4$.

\textit{Claim 1.} If $|A|=8$, then $A$ may be isolated by $\mathcal{F}_3$.

\textit{Proof of claim 1.} Without loss of generality, let $A=P\setminus \{a_1\}$. By (\ref{cond:C}) there are two $a_2$-cross-lines $l_2,l_2'$ and two $a_3$-cross-lines, $l_3,l_3'$, all of which are distinct. Without loss of generality, after some relabelling, we have that $l_3\cap P=\{a_3,b_3,c_3\}$ and $l_3'\cap P=\{a_3,b_2,c_2\}$. At least one of the $a_2$-cross-lines does not intersect both of these sets, otherwise $\{b_2,b_3,c_2,c_3\}$ would not be a pairing four, contradicting \Cref{no-non-pairing-fours}.

So, perhaps after renumbering, we may pick an $a_2$-cross-line $l_2$ disjoint in $P$ from $l_3$. If $l_2 \cap l_3' \cap P =\es$ as well, we are done since $A\setminus (l_2\cup l_3)$ must be entirely contained in $l_3'$ and thus we can isolate $A$ using $l_2,l_3,l_3'$ (none of these lines pass through $a_1$ since $P$ does not contain four collinear points). So, suppose otherwise: without loss of generality suppose that $c_2\in l_2$, so that $l_2\cap P=\{a_2,b_1,c_2\}$. By (\ref{cond:B}), $c_2\notin l_2'$. So, if $l_2'$ intersects both $l_3$ and $l_3'$ in $P$, we have that $b_2,c_3\in l_2'$. In this case, $\{a_2,a_3,c_2,c_3\}$ is not a pairing four, contradicting \Cref{no-non-pairing-fours}. So, this cannot be the case. By the same reasoning as above, we are done if $l_2'$ is disjoint in $P$ from both $l_3$ and $l_3'$. We are left with two cases to consider.

\noindent\textbf{Case I:} $l_2'$ intersects $l_3'$ in $P$. Now, $l_2'$ must pass through one of $b_2,b_3$, since $b_1 \in l_2$, and does not intersect $l_3$. Since $b_3\in l_3$, we must have that $b_2\in l_2'$. Thus, since $c_3\in l_3$ and $\{b_2,c_2\} \sub l_3'$, $l_2'\cap P= \{a_2,b_2,c_1\}$. We may therefore isolate $A$ using $l_3,l_2,l_2'$.

\noindent\textbf{Case II:} $l_2'$ intersects $l_3$ in $P$. Symmetrically to Case I, we have that $l_2'\cap P= \{a_2,b_3,c_1\}$. Now, if $a_1,b_2,c_1$ are not collinear, then we are done, isolating $A$ using $l_2,l_3,l_{b_2 c_1}$. If they are collinear, then, by (\ref{cond:B}), the $a_1$-cross-lines must be as follows: $l_1=\{a_1,b_2,c_1\}$, $l_1'=\{a_1,b_1,c_3\}$. This forces $P$ to be an $\mathsf{X}$-configuration in which (\ref{cond:C}) fails. So, we are done.\hfill$\blacksquare_{\textit{Claim }1}$

\textit{Claim 2.} If $5\leq |A|\leq 7$ and there are three points in $A$ collinear, $A$ may be isolated by $\mathcal{F}_3$.

\textit{Proof of claim 2.} If $|A|\leq5$, the claim is trivial. If there are two disjoint sets of three points collinear in $A$ we are also done. So, if $|A|=7$ we are done by \Cref{4avoid2}. It remains to consider the case that $|A|=6$ and there are not two disjoint sets of three collinear points.

If some $L_i\cap P$ (without loss of generality suppose $i=1$) is contained in $A$, then we may assume that there are $x,y\in A$ with $x\in L_2$ and $y\in L_3$. Then we may isolate $A$ using $L_1$, $l_{xy}$ and a line through the remaining point (note that $l_{xy}$ intersects with at most some element of $L_1\cap P$, by assumption. 

So, we may suppose that this is not the case, and that there is a cross-line $l$ with $l\cap P\subseteq A$ (three collinear points in $P$ must lie on some $L_i$ or a cross-line). Without loss of generality, suppose that $l\cap P= \{a_1,b_1,c_1\}$, $A\setminus l= \{a_2,b_2,c_2\}$ and that $a_2,b_2,c_2$ do not lie on a common cross-line. We are done unless the following sets lie on cross-lines: $\{a_2,b_2,c_3\}$, $\{a_2,b_3,c_2\}$, $\{a_3,b_2,c_2\}$. In this case, by (\ref{cond:B}), we see that $\{a_1,b_3,c_3\}$ must also lie on a cross-line (there are two $a_1$-cross-lines and neither may pass through $b_2$ or $c_2$). But then $\{b_2,b_3,c_2,c_3\}$ is not a pairing four, contradicting \Cref{no-non-pairing-fours}.\hfill$\blacksquare_{\textit{Claim }2}$

\textit{Claim 3.} If $5\leq |A|\leq 7$ and there are no three points in $A$ collinear, $A$ may be isolated by $\mathcal{F}_3$.

\textit{Proof of claim 3.} It is easy to see that there are no sets of size $7$ with no three points collinear. So, suppose first that $|A|=6$ and assume $a_1\in A$. We may suppose that $\{a_1,b_1,c_1\}$ and $\{a_1,b_2,c_2\}$ lie in cross-lines. Since there are no three points in $A$ collinear, $L_1\cap P \nsubseteq A$. So, without loss of generality, we may assume that $a_3\notin A$. Since $A$ may contain at most two points from each of $L_2,L_3$ and the $a_1$-cross-lines, we may further assume that $b_1,c_2\notin A$ and $a_2,b_2,b_3,c_1,c_3\in A$. Given this we must have that $\{a_2,b_3,c_1\}$ and $\{a_2,b_2,c_3\}$ are sets of non-collinear points. It follows that the following define the $a_2$-cross-lines: $\{a_2, b_1, c_3\}$ and $\{a_2, b_3, c_2\}$. This leaves two cases, since $a_3$ lies on exactly two cross-lines.

\noindent\textbf{Case I:} $\{a_3,b_3,c_1\}$ and $\{a_3,b_2,c_3\}$ lie on cross-lines. We may check that in this case $P$ is an $\mathsf{X}$-configuration, and therefore (\ref{cond:C}) fails.

\noindent\textbf{Case II:} $\{a_3,b_3,c_3\}$ and $\{a_3,b_2,c_1\}$ lie on cross-lines. Given this full characterisation of cross-lines we may see which elements of $A$ pair inside $A$, namely (at least): $\{a_1,c_3\}$, $\{a_2,b_2\}$ and $\{b_3, c_1\}$. This allows us to isolate $A$. 

Now suppose that $|A|=5$. Since no three points in $A$ collinear, we must choose one point from one $L_i$ and 2 points from each of the other two $L_i$. Suppose that $a_1\in A$ and $a_2,a_3\notin A$. Then, assuming (without loss of generality) that $\{a_1,b_1,c_1\}$ and $\{a_1,b_2,c_2\}$ lie on cross-lines, we must have (without loss of generality) that $b_1,c_2\notin A$ so that $b_2,c_1,b_3,c_3\in A$. Now, either $\{c_1,b_2\}$ or $\{c_1,b_3\}$ pair inside $A$ (otherwise $c_1$ is in three cross-lines). For the same reason, $a_1$ pairs with $c_3$ in $A$. So, we are done in either case.\hfill$\blacksquare_{\textit{Claim }3}$

This exhausts all cases, so $A$ may be isolated by $\mathcal{F}_3$ for any $A\subseteq P$.
\end{proof}

\begin{proposition}[Necessity]\label{necessity-B}
If $P\subseteq \mathbb{R}^2$ with $|P|=9$ is shattered by $\F_3$ and does not contain four collinear points then $P$ must satisfy (\ref{cond:0}), (\ref{cond:B}), and (\ref{cond:C}).
\end{proposition}

\begin{proof}
Suppose that $P$ is shattered by $\mathcal{F}_3$ and does not contain four collinear points. It is again trivial that (\ref{cond:0}) holds, so we may partition $P$ into the following sets of collinear points where no point from one set is collinear with any two from another: $A:=\{a_1,a_2,a_2\}$, $B:=\{b_1,b_2,b_3\}$, and $C:=\{c_1,c_2,c_3\}$ (assume these lie on $L_1,L_2,L_3$ respectively). If $x$ is a point in $P$ then to isolate $P\sm \{x\}$ we must find lines $l_x$ and $l_x'$ with $x\notin l_x\cup l_x'$, $l_x\cap l_x'\cap P=\emptyset$, $|l_x\cap P|=|l_x'\cap P|=3$ and $P\sm (l_x\cup l_x')$ not a set of collinear points - call this requirement ($\dag$). We see that this is not possible if $L_i \in \{l_x,l_x'\}$ for some $i\in \{1,2,3\}$. Thus, we must have that $l_x,l_x'$ are cross-lines with respect to $\{L_1,L_2,L_3\}$.

\textit{Claim 1.} There are six distinct cross-lines.

\textit{Proof of Claim 1.} It is clear that each point in $P$ must be an element of a cross-line. For example, choose the point $a_1$. Then we may apply the above argument with $x=a_2$ and we see that $a_1$ is in $l_{a_2}$ or $l_{a_2}'$. Therefore, we easily see that we must have at least four cross-lines. Indeed, if there are only three cross-lines and all points are in some cross-line, we have that the complement of any two cross-lines in $P$ is a set of three collinear points. So, requirement ($\dag$) fails. So, we go progressively by cases.

\noindent\textbf{Case I:} Suppose we have only four cross-lines. Then we may suppose (without loss of generality) that $a_1$ is in two cross-lines, $l_1,l_1'$ while $a_2,a_3$ are each in just one, $l_2$ and $l_3$ respectively. Now, we must have that $\collin(P\setminus (l_2\cup l_3))\leq 3$, by requirement ($\dag$), so there are two cases to consider.

\textbf{Subcase I(a):} $l_j\cap l_1 \cap P \neq \es$ and $l_j\cap l_1'\cap P\neq \es$ for some $j\in \{2,3\}$. Suppose without loss of generality that $j=2$. But then we may let $x=a_3$ and note that the requirement ($\dag$) fails.

\textbf{Subcase I(b):} Without loss of generality, $l_1\cap l_2 \cap P \neq \es$ and $l_1'\cap l_3\cap P \neq \es$. Then choose $x\in l_1\cap l_2 \cap P$. Requirement ($\dag$) fails for this choice since the only two cross-lines not containing $x$ are not disjoint in $P$.

So, there are at least five cross-lines.

\noindent\textbf{Case II:} Suppose we have only five cross-lines. There are two subcases to consider.

\textbf{Subcase II(a):} Suppose some element of $P$ is a 3-node. Without loss of generality, suppose that $a_1$ is a 3-node. Then each point in $C$ is in an $a_1$-cross-line. 

Suppose for contradiction that $c_1$ is a 3-node. We may suppose without loss of generality that $\collin(\{a_1,b_1,c_1\})=3$ and therefore that $\collin(\{a_1,b_2,c_1\})=\collin(\{a_1,b_3,c_1\})=2$. But then $\{a_1,c_1,b_2,b_3\}$ is not a pairing four, since each of $b_2,b_3$ are in an $a_1$cross-line and a distinct $c_1$-cross-line. So, no point in $C$ is a 3-node.

Therefore, we have without loss of generality that $c_1,c_2$ are 2-nodes and $c_3$ is a 1-node. Now, each element of $B$ must be in a $c_1$- or $c_2$-cross-line. We see this by supposing otherwise, so that (without loss of generality) $b_3$ is not in either a $c_1$- or a $c_2$-cross-line. Then $b_1,b_2$ are contained in two such cross-lines each and $\{c_1,c_2,b_1,b_2\}$ is not a pairing four, which is a contradiction. So, letting $l$ be the $c_3$-cross-line, we may assume that $l\cap B=\{b_1\}$ and that $b_1$ is in a $c_1$-cross-line. Since $a_1$ is a 3-node, we must have that $l\cap A=\{a_1\}$ (since some cross-line passes through $a_1$ and $c_1$, so it must be the unique one). We therefore see that $a_1$ is contained in the $c_1$-cross-line that does not pass through $b_1$. But then we may let $x=c_2$ and conclude that the requirement ($\dag$) fails, since every $c_1$-cross-line intersects the unique $c_3$-cross-line.

\textbf{Subcase II(b):} Suppose there is no 3-node, so that each of $A$,$B$ and $C$ consists of two 2-nodes and one 1-node. Suppose $a_1,a_2$ are 2-nodes and $a_3$ is a 1-node. To avoid failure of the requirement ($\dag$) for $x=a_1$, we see that each of the $a_1$-cross-lines must be intersected by some $a_2$- or $a_3$-cross-line (not necessarily the same one).

Suppose first that no $a_2$- or $a_3$-cross-line intersects both $a_1$-cross-lines. Without loss of generality we may assume that cross-lines pass through the following sets: $\{a_1,b_1,c_1\}$, $\{a_1,b_2,c_2\}$. On these assumptions, we must have (without loss of generality) that cross-lines pass through $\{a_2,b_3,c_2\}$ and $\{a_2,b_1,c_3\}$. Again using the assumption, we must therefore have that a cross-line passes through either $\{a_3,b_3,c_1\}$ or $\{a_3,b_2,c_3\}$. In either case, the requirement ($\dag$) fails if we set $x=a_1$. We see this since in either case, there is only one $a_2$-cross-line that is disjoint from the $a_3$-cross-line in $P$ and the complement of these two lines is an $a_1$-cross-line.

So, now suppose that some cross-line intersects both $a_1$-cross-lines. 

Arguing as in Subcase I(a) (in particular, concluding that ($\dag$) fails for $x=a_2$), we see that this cross-line cannot be the unique $a_3$-cross-line, so without loss of generality we may assume that the following set lies on a cross-line: $\{a_2,b_2,c_1\}$. To avoid producing subset of size four that is not pairing, we must have that a cross-line passes through $\{a_2,b_3,c_3\}$ as well. (If $\collin(\{a_2,b_3,c_2\})=3$, then $\{a_1,a_2,c_1,c_2\}$ is not a pairing four, and if $\collin(\{a_2,b_1,c_3\})=3$, then $\{a_1,a_2,b_1,b_2\}$ is a non-pairing four. These are the only other options for the second $a_2$-cross-line.) Now consider where the $a_3$-cross-line, call it $l$, may pass. It may not intersect $b_2$ or $c_1$ since there are no 3-nodes. It may not pass through both $b_1,c_2$ since this would render $\{b_1,b_2,c_1,c_2\}$ not a pairing four. So, if it passes through $b_1$ it must pass through $c_3$ and if it passes through $b_3$ it must pass through $c_2$ (note that $b_3,c_3$ are already on an $a_2$-cross-line). In the first case taking $x=c_3$ and in the second case taking $x=b_3$ forces the requirement ($\dag$) to fail. In the first case, with $x=c_3$, we see that any $c_1$-cross-line intersects the unique $c_2$-cross-line. In the second case, with $x=b_3$, we see that any $b_2$-cross-line intersects the unique $b_1$-cross-line.

So, there must be at least six cross-lines. \hfill$\blacksquare_{\textit{Claim }1}$

We then show that there is no 3-node, from which we may conclude (\ref{cond:B}). Suppose for contradiction that $a_1$ is a 3-node. Then without loss of generality we may assume that $a_2$ is a 2-node or a 3-node, so that $b_1,b_2$ (without loss of generality) are each contained in $a_2$-cross-lines. But then $\{a_1,a_2,b_1,b_2\}$ is not a pairing four, which is a contradiction. So, (\ref{cond:B}) holds.

Now assume that (\ref{cond:0}), (\ref{cond:B}) hold but (\ref{cond:C}) fails. Then for any $y\in P$ there is no cross-line that intersects both $y$-cross-lines once. So, we may assume without loss of generality the following sets lie on cross-lines: $\{a_1,b_1,c_1\}$, $\{a_1,b_2,c_2\}$, $\{a_2,b_3,c_2\}$, $\{a_2,b_1,c_3\}$ (if we had $\{a_2,b_3,c_3\}$ lying on a cross-line, the other $a_2$-cross-line would grant (\ref{cond:C}). By our assumption, we may assume that $\{a_3,b_2,c_3\}$, $\{a_3,b_3,c_1\}$ also lie on cross-lines. Since there may be no more cross-lines, this set is an $\mathsf{X}$-configuration, as in \Cref{X-configuration}. In this case, requirement ($\dag$) fails for any $x\in P$. So, (\ref{cond:C}) holds as well.
\end{proof}

As an easy consequence of either Case A or Case B we obtain the following corollary.

\begin{corollary}\label{vc-3-lines}
    $\VC(\F_3) = 9$.
\end{corollary}
\begin{proof}
    We already know that $\VC(\F_3)\leq 9$, so, by the previous results in this section, it suffices to show that there exists a set $P\subseteq\R^2$ of nine points satisfying either of the following sets of axioms.\begin{enumerate}
        \item[(A)] $P$ has four collinear points and satisfies (\ref{cond:0}), (\ref{cond:b}), and (\ref{cond:c}).
        \item[(B)] $P$ has no four points collinear and satisfies (\ref{cond:0}), (\ref{cond:B}), and (\ref{cond:C}).
    \end{enumerate}
    
    Such sets were produced in \Cref{ex:set1,ex:set2}, respectively.
\end{proof}

\section{Counting Shatterable Sets}\label{sec:counting}

As discussed in \Cref{sec:shatter-isomorphism}, determining the exact values of $s_k(X,\mathcal{S})$, for a fixed set-system $(X,\mathcal{S})$, is, in general, a difficult problem. In this section, we prove \Cref{thm:classification-intro}, namely that for the family of lines in $\mathbb{R}^2$, $s_{2} = 2$ (\Cref{iso2}) and $s_3=5$ (\Cref{iso3}). An explicit list of representatives for each isomorphism type can be found in \Cref{app:pictures}. First, to facilitate notation, we specialise the general notation of set-system structures introduced in \Cref{sec:shatter-isomorphism} to this case. 

Let $P=\{p_0,\dots,p_{n-1}\}$ be a finite set of points in $\mathbb{R}^2$. As introduced in \Cref{sec:terminology-notation}, we write $\Scal_{P}^2 = \{l_{a,b}:a,b\in P, a\neq b\}$ for the set of all lines which go through (at least) two points in $P$. Working with the lines in $\mathcal{S}_P^2$ together with, for each $p\in P$, a unique line $l_p$ such that $l_p\cap P = \{p\}$ is precisely working with the quotient $\sim_P$, as described in the beginning of \Cref{sec:shatter-isomorphism}. For the moment, to simplify notation, we may ignore the lines $\{l_p:p\in P\}$, as they play no part in the argument below. 

Define a language $\mathcal{L}_P$ with unary predicates $K_i$, for each line $K_i\in\Scal_P^2$. We can view the tuple $(P;K_0,\dots,K_{m-1})$ as an $\mathcal{L}_P$-structure, by interpreting the relation symbols in the obvious manner.  

\begin{proposition}\label{iso2} 
There are two isomorphism types of sets of size five shattered by $\mathcal{F}_2$. 
\end{proposition}

\begin{proof}
Suppose $P\subseteq \mathbb{R}^2$ has $|P|=5$ and $P$ is shattered by $\F_2$. Referring to \Cref{prop:F_2-shattering}, we see that there is a set of three collinear points, $\{a_1,a_2,a_3\}\subseteq P$, lying on a line, $L_1$. There are then two cases to consider.

\textbf{Case I:} Suppose there is a unique set of three collinear points in $P$. Then $\mathcal{L}_P=\{K_1\}$. $(P, L_1)$ is an $\mathcal{L}_P$-structure as above. Suppose that $P'\subseteq \mathbb{R}^2$ has $|P'|=5$, is shattered by $\F_2$ and has a unique set of three points collinear, $\{a_1',a_2',a_3'\}$, lying on a line, $L_1'$. Clearly, $\mathcal{L}_P=\mathcal{L}_{P'}$. Further, a bijection $\phi: P\rightarrow P'$ with $\phi(a_i)=a_i'$ for each $i\in \{1,2,3\}$ is an isomorphism of $\mathcal{L}_P$-structures. 

\textbf{Case II:} There are at least two distinct sets of three collinear points in $P$. Clearly, there are exactly two distinct sets of three collinear points that intersect at exactly one point. Arguing as above, this is enough to determine an isomorphism class.

A set of five points $P$ which is shattered by $\F_2$ must satisfy either Case I or Case II. So, if $P$ and $P'$ satisfy the same case, they are isomorphic. Evidently, if $P$ and $P'$ satisfy different cases, they are not isomorphic (in particular, $\mathcal{L}_P\neq \mathcal{L}_{P'}$). By \Cref{app:pictures}, these isomorphism classes are non-empty. This establishes the result.
\end{proof}

\begin{proposition}\label{iso3} 
There are five isomorphism types of sets of nine points that are shattered by $\F_3$.
\end{proposition}

\begin{proof}
Suppose $P$ is a set of nine points shattered by $\F_3$ then we construct an exhaustive set of five cases for $\mathcal{L}_P$ and the associated $\mathcal{L}_P$-structure. We then show that, if $P$ and $P'$ are shattered by $\F_3$, and thus contain four collinear points and satisfy $\{$(\ref{cond:0}), (\ref{cond:b}), (\ref{cond:c})$\}$ or do not contain four collinear points and satisfy $\{$(\ref{cond:0}), (\ref{cond:B}), (\ref{cond:C})$\}$, then if $P$ and $P'$ satisfy the same case, they are isomorphic. We then show that if $P$ and $P'$ satisfy different cases, they are not isomorphic. Finally, we show that there is at least one such $P$ satisfying each case.

\textit{Claim 1.} There are at most four isomorphism classes of sets $P\subseteq \mathbb{R}^2$ that have $|P|=9$, contain four collinear points and satisfy $\{$(\ref{cond:0}), (\ref{cond:b}), (\ref{cond:c})$\}$. These are given by Subcases I(a), I(b), II(a) and II(b).

\textit{Proof of Claim 1.} By assumption $P$ contains four collinear points, say $A:=\{a_1,a_2,a_3,a_4\}$. Let $P\setminus A:= \{b_1,...,b_5\}$ and assume that $\{b_1,b_2,b_3\}$ are collinear. These sets give us lines $l_A$ and $l_B$. We now argue by cases.

\noindent\textbf{Case I:} $|l_B\cap P|=3$. So by condition (\ref{cond:c}), we have ordinary 3-lines $k_i$ corresponding to each $a_i$, that by \Cref{lem:set1-preparatory}(\ref{O3line-intersect-3collin}) must pass through some $b_1,b_2,b_3$ and some $b_4,b_5$. By \Cref{cond:d} each of the $b_j$, for $j\in \{1,2,3,4,5\}$, must be in at most two $k_i$, for $i\in \{1,2,3,4\}$. Therefore, each of $b_4,b_5$ are in exactly two such $k_i$. By \Cref{cond:d} again, there are the following two possibilities, up to renumbering.

\textbf{Subcase I(a):} Each of $b_1,b_2$ are in two $l_i$ are $b_3$ is in none. Suppose without loss of generality that the following is the case: $k_1\cap P= \{a_1,b_1,b_5\}$, $k_2\cap P= \{a_2,b_1,b_4\}$, $k_3\cap P= \{a_3,b_2,b_5\}$, and $k_4\cap P= \{a_4,b_2,b_4\}$.

Now, if $P$ satisfies the assumptions of I(a), the lines passing through at least three points in $P$ are the following: $l_A, l_B, k_1,k_2,k_3,k_4$. Indeed, suppose $l$ is a line passing through at least three points in $P$. Then, if $l\cap \{a_1,a_2,a_3,a_4\}\neq \es$, then $l=l_A$ or $l\in O_{\geq3}(l_A)$. If $l\in O_{\geq3}(l_A)$, then $l\cap \{b_4,b_5\}\neq \es$. So, by \Cref{cond:d}, $l\in \{k_1,k_2,k_3,k_4\}$. If $l\cap \{a_1,a_2,a_3,a_4\}= \es$, then $l=l_B$ or $|l\cap \{b_1,b_2,b_3\}|=1$. If $|l\cap \{b_1,b_2,b_3\}|=1$, then $\{b_4,b_5\}\subseteq l$. So, since $l\cap \{a_1,a_2,a_3,a_4\}= \es$, $l\cap \{b_1,b_2,b_3\}=\{b_3\}$. But, in this case, $\{b_4,b_5,b_1,b_2\}$ is not a pairing four, which contradicts (\ref{cond:b}). 

So, $\mathcal{L}_P=\{K_1,...,K_6\}$ and $(P,l_A, l_B, k_1,k_2,k_3,k_4)$ is an $\mathcal{L}_P$-structure as in the definition. Suppose $P'\subseteq \mathbb{R}^2$ has $|P|=9$, is shattered by $\F_3$, and satisfies the assumptions of I(a). Then we may label elements of $P'$ analogously: $P'=\{a_1',a_2',a_3',a_4',b_1',b_2',b_3',b_4',b_5'\}$. The lines passing through at least three points may be labelled analogously too: $l_A', l_B', k_1',k_2',k_3',k_4'$. So, $\mathcal{L}(P)=\mathcal{L}(P')$. Then, the map $\phi: P\rightarrow P'$ defined by $\phi(a_i)=a_i'$ for each $i\in \{1,2,3,4\}$ and $\phi(b_j)=b_j'$ for each $j\in \{1,2,3,4,5\}$ is an isomorphism of $\mathcal{L}_P$-structures, since the labelling determines exactly which lines each point is on.

\textbf{Subcase I(b):} Each of $b_1,b_3$ are in one $k_i$ and $b_2$ is in two. Suppose without loss of generality that the following is the case: $k_1\cap P= \{a_1,b_1,b_4\}$, $k_2\cap P= \{a_2,b_2,b_5\}$, $k_3\cap P= \{a_3,b_2,b_4\}$, and $k_4\cap P= \{a_4,b_3,b_5\}$.

Now, if $P$ satisfies the assumptions of I(b), the lines passing through at least three points in $P$ are the following: $l_A, l_B, k_1,k_2,k_3,k_4$. The proof proceeds exactly as above, except that the case $\{b_4,b_5\}\subseteq l$ and $l\cap \{a_1,a_2,a_3,a_4\}= \es$ is ruled out immediately, since any $l$ with $\{b_4,b_5\}\subseteq l$ is such that $l\in O_{\geq3}(l_A)$ in this case. Clearly, we may conclude that this subcase determines an isomorphism class exactly as in the previous subcase. 

\noindent\textbf{Case II:} $|l_B\cap P|=4$. In this case, by \Cref{basic-fact}(\labelcref{no-2-n+1-disjoint}), we may assume that $l_B\cap P= \{a_1,b_1,b_2,b_3\}$.  Then by condition (\ref{cond:c}) we must have $k_i\in O_{\geq3}(l_A)$ corresponding to $a_i$ for each $i\in \{2,3,4\}$. Each of these $l_i$ must pass through at least one of $b_4,b_5$, so we have the following cases, that, by (\ref{cond:d}), are all we need to consider up to renumbering.

\textbf{Subcase II(a):} $b_4$ is in two $k_i$ and $b_5$ is in one. In particular, $|k_i\cap \{b_4,b_5\}|=1$ for each $i\in \{2,3,4\}$. We may suppose without loss of generality that $b_4\in k_2\cap k_3$ and $b_5\in k_4$. We may assume without loss of generality that $k_2\cap P= \{a_2,b_4,b_1\}$ and $k_3\cap P= \{a_3, b_4,b_2\}$. By \Cref{cond:d}, $b_1,b_2\notin k_4$. So, $k_4\cap P= \{a_4,b_5,b_3\}$. 

Now, if $P$ satisfies the assumptions of II(a), the lines passing through at least three points in $P$ are the following: $l_A, l_B,k_2,k_3,k_4$. Indeed, suppose $l$ is a line passing through at least three points in $P$. If $l\cap \{a_1,a_2,a_3,a_4\}\neq \es$, then $l=l_A$ or $l\in O_{\geq3}(l_A)$. If $l\in O_{\geq3}(l_A)$, then $l\cap \{b_1,b_2,b_3,b_4\}\neq \es\}$. Since all of these points are in two of $l_B,k_2,k_3,k_4$, by \Cref{cond:d}, $l\in \{l_B,k_2,k_3,k_4\}$. If $l\cap \{a_1,a_2,a_3,a_4\}= \es$, then $l\cap \{a_1,b_2,b_3,b_4\} \neq \es$. So, we may carry out the same reasoning again.

As in previous cases, this labelling determines an isomorphism between structures satisfying the assumptions of II(a). 

\textbf{Subcase II(b):} Both $b_4$ and $b_5$ are in two $k_i$. In particular, $|k_j\cap \{b_4,b_5\}|=2$ for some $j\in \{2,3,4\}$. Further, $|k_i\cap \{b_4,b_5\}|=1$ for each $i\in \{2,3,4\}\setminus \{j\}$. We may suppose without loss of generality that $b_4\in k_2\cap k_3$ and $b_5\in k_2\cap k_4$. We may assume without loss of generality that $k_2\cap P\supseteq \{a_2,b_4,b_5\}$, $k_3\cap P= \{a_3,b_4,b_1\}$ and $k_4\cap P= \{a_4,b_5,b_2\}$. By \Cref{O3line-intersect-3collin}, $k_2\notin O_3(l_A)$ since in that case $(k_2\cap \{b_1,b_2,b_3\})\cap P =\es$. So, $k_2\in O_4(l_A)$. Therefore, $k_2\cap \{b_1,b_2,b_3\}\neq \es$. By \Cref{cond:d}, $b_3\in k_2$, so $k_2\cap P= \{a_2,b_4,b_5,b_3\}$. 

By an almost identical argument as the one for Subcase II(a), the lines passing through at least three points in $P$ are: $l_A, l_B,k_2,k_3,k_4$.

As in previous cases, this labelling determines an isomorphism between structures satisfying this subcase.
\hfill$\blacksquare_{\textit{Claim }1}$

\textit{Claim 2.} There is at most one isomorphism class of sets $P\subseteq \mathbb{R}^2$ that have $|P|=9$ do not contain four collinear points and satisfy $\{$(\ref{cond:0}), (\ref{cond:B}), (\ref{cond:C})$\}$. This is given by Case III.

\textit{Proof of Claim 2.} We now place ourselves in the following case.

\textbf{Case III:} $P$ does not contain four collinear points and satisfies $\{$(\ref{cond:0}), (\ref{cond:B}), (\ref{cond:C})$\}$. Then we may choose $L_1,L_2,L_3$ with $L_1\cap P = \{a_1,a_2,a_3\}$, $L_2\cap P = \{b_1,b_2,b_3\}$ and $L_3\cap P = \{c_1,c_2,c_3\}$ (these sets pairwise disjoint). We may fix exactly six distinct cross-lines: $l_{a_1}, l_{a_1}', l_{a_2}, l_{a_2}', l_{a_3}, l_{a_3}'$ such that $l_{a_i}, l_{a_i}'$ are $a_i$-cross-lines for each $i\in \{1,2,3\}$. Suppose without loss of generality that $l_{a_1}\cap P= \{a_1,b_1,c_1\}$ and $l_{a_1}'\cap P= \{a_1,b_2,c_2\}$. We may also suppose without loss of generality that $a_1$ witness the $y\in P$ from (\ref{cond:C}). Without loss of generality, therefore, we may assume that $l_{a_2}\cap P= \{a_2,b_2,c_1\}$. Then we follow the same reasoning as in the proof of \Cref{necessity-B}. Since all fours are pairing in $P$, we see that $l_{a_2}'\cap P =\{a_2,b_3,c_3\}$. From here, again since all fours are pairing, $l_{a_3}\cap P =\{a_3,b_3,c_2\}$ and $l_{a_3}'\cap P =\{a_3,b_1,c_3\}$.

Clearly, since every element of $P$ is a 2-node, the lines passing through at least three points in $P$ are exactly: $L_1, L_2, L_3, l_{a_1}, l_{a_1}', l_{a_2}, l_{a_2}', l_{a_3}, l_{a_3}'$. 

As in previous cases, this labelling determines an isomorphism between structures satisfying this subcase.\hfill$\blacksquare_{\textit{Claim }2}$

\textit{Claim 3.} If $A, A'\subseteq \mathbb{R}^2$ are such that $|A|=|A'|=9$, both are shattered by $\F_3$, and each satisfies different cases from I(a), I(b), II(a), II(b) and III, then $A$ and $A'$ are not isomorphic.

\textit{Proof of Claim 3.} If $P, P'\subseteq \mathbb{R}^2$ are such that $|P|=|P'|=9$, both are shattered by $\F_3$, and each satisfy different cases from I(a), I(b), II(a), II(b) and III, then $P$ and $P'$ are not isomorphic.

\textit{Proof of Claim 3.} This is clear by inspecting the conditions.  \hfill$\blacksquare_{\textit{Claim }3}$

Since each of these five isomorphism classes are non-empty by \Cref{app:pictures}, it follows that there are exactly five isomorphism classes, as required.
\end{proof}

\section{Higher-Dimension Analogues of \texorpdfstring{\Cref{thm:vc-intro}}{Theorem A}}\label{sec:higher-arity}

As discussed in the Introduction, we conclude the main body of this paper by showing that \Cref{thm:vc-intro} can easily be transferred to higher dimensions. First, we introduce some notation.

\textbf{Notation.} Let $n\in\Nbb_{>0}$ and $m\in\Nbb$, where $m\leq n$. We write $\Abb_{n,m}$ for the set of $m$-dimensional affine subsets of $\Rbb^n$.
\[
    \Abb_{n,m}:= \{V\sub \R^n: V\text{ is affine, }\dim(V)=m\}.
\]

In this notation, the set of all points in $\Rbb^2$ is just $\Abb_{2,0}$ and the set $L$ of all lines in $\Rbb^2$ just $\Abb_{2,1}$. In particular, the family $\F_k$ that we have considered in the previous sections of this paper (for $k=2,3$) is precisely $\Abb_{2,1}^{\cup k}$.

The main result in this section is the following.

\begin{theorem}\label{thm:linear-alg}
    Let $n\in\Nbb_{\geq 2}$ and $k\in\Nbb$. 
    \[
        \VC\left(\Abb_{n,n-2},\Abb_{n,n-1}^{\cup k}\right)= \VC\left(\F_k\right).
    \]
\end{theorem}

Before proving the theorem above, a rather important clarification is needed. Indeed, the  reader may have noticed that $(\Abb_{n,n-2},\Abb_{n,n-1})$, and in general $(\Abb_{n,n-2},\Abb_{n,n-1}^{\cup k})$, is not formally a set system, in the sense defined in \Cref{sec:intro}. This is because $\Abb_{n,n-1}$ is not a subset of $\pow(\Abb_{n,n-2})$. However, this is only a technical matter, which can be easily (and intuitively) remedied by identifying $\Abb_{n,n-1}$ with the subset of $\pow(\Abb_{n,n-2})$ which consists only of sets of $(n-2)$-dimensional affine subspaces of $\Rbb^n$ whose union is an $(n-1)$-dimensional affine subspace of $\Rbb^n$.


\begin{proof}
    Recall that for a set system $(X,\Scal)$ and $Y\subseteq X$ we define an equivalence relation $\sim_Y$ on $\Scal$.
    \[S\sim_A S' \text{ if and only if } S\cap A = S'\cap A,\]
    for each $S,S'\in\Scal$. In light of \Cref{lem:k-fold-unions}, to prove \Cref{thm:linear-alg} it suffices to show the following claims.
    \begin{enumerate}
        \item For all finite $A\subseteq \Abb_{n,{n-2}}$ there is a finite $B\subseteq P$ such that $(A,\Abb_{n,n-1}/_{\sim_A)}$ and $(B,L/_{\sim_B})$ are shatter isomorphic.
        \item For every finite $B\subseteq P$ there is a finite $A\subseteq\Abb_{n,n-2}$ such that $(A,\Abb_{n,n-1}/_{\sim_A})$ and $(B,L/_{\sim_B})$ are shatter isomorphic.
    \end{enumerate}
    Clearly, the second point is trivial, so we may only focus on the first. 

    Let $A\subseteq \Abb_{n,n-2}$ be a finite set of $(n-2)$-dimensional affine subsets of $\Rbb^n$. By definition, $A$ is of the form $\{S_1+p_1,\dots,S_m+p_m\}$, where $S_i\subseteq\Rbb^n$ is an $(n-2)$-dimensional subspace and $p_i\in\Rbb^n$, for $i\leq m$. Let $U\subseteq\Rbb^n$ be an $(n-1)$-dimensional subspace such that $S_m\not\subseteq U$ (the existence of such a $U$ follows from an easy linear algebra argument). It is easy to see that $S_i+p_i\not\subseteq U$, for each $i\leq m$. On the other hand, we must have that $(S_i+p_i)\cap U\neq \emptyset$, for each $i\leq m$. 
    
    \emph{Claim.}
        There is a translation $U'=U+q$ of $U$ such that the following conditions hold.
        \begin{enumerate}
            \item For each $(S+p)\in A$ we have that $\es \subsetneq (S+p)\cap U' \subsetneq U'$.
            \item If $S+p, S'+p'\in A$ are distinct, then $(S+p)\cap U' \neq (S'+p') \cap U'$. 
            \item For any $V\in \Abb_{n,n-1}$, if $|\{S+p\in A: (S+p)\sub V\}| \geq 2$, then:
            \[
                \{S+p\in A: (S+p)\cap U' \sub V\} = \{S+p\in A: (S+p)\sub V\}.
            \]
        \end{enumerate}

Given the claim above, let $B_{n-1}\subseteq\Abb_{n-1,n-3}$ be the set $B_{n-1} = \{(a+S)\cap U'\}$. It is straightforward to check that the conditions in the claim ensure that we have a shatter-isomorphism between $(A,\Abb_{n,n-1}/_{\sim_A})$ and $(B_{n-1},\Abb_{n-1,n-2}/_{\sim_{B_{n-1}}})$. Repeating the same argument $n-2$ times, we get the result. Thus, it suffices to prove the claim. 

\emph{Proof of Claim.}
    First, we observe that if $U'$ is any translation of $U$, then, for any $S+p\in A$ we have that $S+p\not\subseteq U'$ but $(S+p)\cap U'\neq \emptyset$. This means that condition (1) in the claim will automatically hold for any translation of $U$. Thus, it suffices to show that conditions (2) and (3) hold in some translation of $U$. We will, in fact, show that they each hold in all but finitely many translations of $U$. For (2), suppose that for some $q\in\Rbb^n$ and distinct $S+p,S'+p'\in A$ we have that $(S+p)\cap (U+q) = (S'+p')\cap (U+q)$. A direct dimension calculation shows that this $q$ must be unique. Thus, each pair of distinct elements of $A$ determines at most one translate of $U$ which does not satisfy (2). Since $A$ is finite, it follows that (2) is satisfied by all but finitely many translates of $U$. Similarly, we may show that at most finitely many translates of $U$ do not satisfy (3). Thus, the claim follows. \phantom\qedhere
    {\hfill$\blacksquare_{\textit{Claim}}$ }\qedsymbol

\end{proof}

\appendix
\section{The proof of \texorpdfstring{\Cref{lem:set1-preparatory}}{Proposition}}\label{pf:set1-preparatory}

\prelim*

\begin{proof}[Proof of (\ref{O3line-intersect-3collin})]
Let $l\cap P = \{a_1,a_2,a_3,a_4\}$, $P\setminus l = \{b_1,b_2,b_3,b_4,b_5\}$. Take a set of three collinear points in $P\setminus l$: without loss of generality, let this set be $\{b_1,b_2,b_3\}$. Suppose, without loss of generality that $a_1\in k$. For contradiction, suppose that $K\cap \{b_1,b_2,b_3\}=\es$. So, since $|k\cap P|=3$, $K\cap P = \{a_1,b_4,b_5\}$.

\textit{Claim 1.} There is some $i \in \{2,3,4\}$ and $l_i \in O_4(l)$ such that $l_i \cap l = \{a_i\}$.

\textit{Proof of claim 1.} By (\ref{cond:c}), there are $l_i \in O_{\geq 3}(l)$ such that $a_i \in l_i$ for each $i\in \{2,3,4\}$. By \Cref{cond:d}, there are at most two distinct elements of $O_{\geq 3}(l)$ containing $b_4$ (and similarly $b_5$). So, since $\{b_4,b_5\}\subseteq k$, there is $i \in \{2,3,4\}$ such that $\{b_4,b_5\}\cap l_i =\es$. But then $|l_i\cap \{b_1,b_2,b_3\}|\geq 2$. Since $b_1,b_2,b_3$ are collinear, $\{b_1,b_2,b_3\}\subseteq l_i$. This establishes the claim. 
\hfill$\blacksquare_{\textit{Claim }1}$

So, assume, without loss of generality, that $l_4\in O_4(l)$. Applying (\ref{cond:c}), this time with respect to $l_4$, there are $l_j' \in O_{\geq3}(l_i)$ such that $b_j \in l_j'$ for each $j\in \{1,2,3\}$. Note that since there are no five points collinear in $P$, $l_j' \neq l$ for each $j \in \{1,2,3\}$. So, $|l_j'\cap \{a_1,a_2,a_3,a_4\}| \leq 1$. So, $l_j'\cap \{b_4,b_5\} \neq \es$ for each $j$.

Take $j\in \{1,2,3\}$. If $|l_j'\cap \{b_4,b_5\}|=1$, then $|l_j'\cap \{a_1,a_2,a_3,a_4\}|=1$, so, $l_j'\in O_{\geq3}(l)$. Since $\{b_4,b_5\}\sub k\in O_{\geq3}(l)$, each of $b_4,b_5$ lies on at most one such $l_j'$, using the same reasoning as above (otherwise, one of $b_4,b_5$ would lie on three distinct elements of $O_{\geq3}(l)$, contradicting \Cref{cond:d}. So, there are at most two such $l_j'$.

Thus, there is some $j$ such that $\{b_4,b_5\}\subseteq l_j'$. So, $b_j\in l_j'=K$. But this means that $K\notin O_3(l)$. This is a contradiction.
\end{proof}

\begin{proof}[Proof of (\ref{nobadquadruple})]
Suppose otherwise: for $A=\{a_1,...,a_4\}$ lying on a line $l$, $B=\{b_1,b_2,b_3\}$ and $P\setminus (A\cup B)=\{b_4,b_5\}$ assume $(a_1,a_2,b_1,b_2)$ is a bad quadruple. 

Then each of $l_{a_1b_1}, l_{a_1b_2}, l_{a_2b_1},l_{a_2b_2}$ are in $O_{\geq3}(l)$ and intersect $P\setminus (A\cup B)$ non-trivially. 

\textit{Claim 1.} These lines are pairwise distinct. 

\textit{Proof of Claim 1.} If $l_{a_1b}=l_{a_2b'}$ for any $b,b'\in \{b_1,b_2\}$, then $\{a_1,a_2,a_3,a_4,b\}$ is a set of five collinear points. So, $l_{a_1b}\neq l_{a_2b'}$ for any $b,b'\in \{b_1,b_2\}$. 

Further, if for some $a\in \{a_1,a_2\}$ (without loss of generality choosing $a=a_1$) we have that $l_{a_1b_1}=l_{a_1b_2}$, then we must have that (choosing $b_4\in \{b_4,b_5\}$ without loss of generality) $\{a_1,b_1,b_2,b_4\}$ is a set of collinear points (since each line intersects $P\setminus (A\cup B)$). So, since there are no five points in $P$ collinear, $b_4\notin l_{a_2b_1}\cup l_{a_2b_2}$. But then $b_5\in l_{a_2b_1}\cap l_{a_2b_2}$, so $l_{a_2b_1}=l_{a_2b_2}$. So, $\{b_1,b_2\} \sub l_{a_2b_1}$ and $b_4\in l_{a_2b_1}$ after all, which is a contradiction.\hfill$\blacksquare_{\textit{Claim }1}$

It follows that $b_1,b_2$ and some $b\in \{b_4,b_5\}$ are each contained in at least two lines that are elements of $O_{\geq3}(l)$. Without loss of generality, suppose $b=b_4$. By condition (\ref{cond:c}), there are lines $l_3,l_4\in O_{\geq3}(l)$ passing through $a_3,a_4$ respectively. By \Cref{cond:d}, $(l_3\cup l_4)\cap \{b_1,b_2,b_4\}=\es$, so that $\{b_3,b_5\}\subseteq l_3\cap l_4$. This implies that $l_3=l_4$. In this case, there would be a set of five collinear points, which is a contradiction.    
\end{proof}

\begin{proof}[Proof of (\ref{set1-4collin})]
    Follows immediately from (\ref{cond:0}) and \Cref{prop:F_2-shattering}.
\end{proof}

\begin{proof}[Proof of (\ref{no3collin-leq6})]
We go by cases.

\noindent\textbf{Case I:} Suppose that there are two distinct sets of four collinear points in $P$ : $\{a_1,...,a_4\}$ and $\{a_1,b_1,b_2,b_3\}$ (note that they must intersect by \Cref{basic-fact}). Then we consider two subcases.

\textbf{Subcase I(a):} $a_1\in A$, so we may without loss of generality suppose that $a_1,a_2,b_1\in A$, and that $a_3,a_4,b_2,b_3\notin A$ (since no three points in $A$ are collinear). We may then further assume that the remaining points in $A$, $b_4,b_5$, are in $A$, since $|A|\geq 5$. Now, we take the following lines to isolate $A$: $l_{a_2b_1}, l_{a_1b_4}$ and a line intersecting only $b_5$. These lines do not intersect $\{a_3,a_4,b_2,b_3\}$ since there are no five points in $P$ collinear.

\textbf{Subcase I(b):} Suppose $a_1\notin A$. We know that $|A\cap \{a_1,...,a_4,b_1,b_2,b_3\}|\in \{3,4\}$ since no three points in $A$ are collinear and $|A|\geq 5$.

So, \textit{first assume that $|A\cap \{a_1,...,a_4,b_1,b_2,b_3\}|= 4$}. Without loss of generality, assume that $a_3,a_4,b_2,b_3 \in A$ and that $a_1,a_2,b_1\notin A$. 

If $b_4,b_5\in A$, then we must have that three points in $A$ are collinear. We see this because by condition (\ref{cond:c}) there are lines $l_3,l_4\in O_{\geq3}(l)$, where $l$ is the line passing through $a_1,a_2,a_3,a_4$, with $a_i\in l_i$ for each $i\in \{3,4\}$. Both $l_3,l_4$ must intersect $\{b_1,b_2,b_3\}$ by (\ref{O3line-intersect-3collin}). By \Cref{cond:d}, at most one of these two lines passes through $b_1$. So, we may assume without loss of generality that $l_4$ does not pass through $b_1$. Therefore, $|l_4\cap \{b_4,b_5,b_2,b_3\}|\geq 2$ and $a_4\in l_4$, so $|l_4\cap A|\geq 3$. 

We may therefore suppose without loss of generality that $b_5\notin A$. Since $|A|\geq 5$, we may suppose that $b_4\in A$. But then either $b_5\notin l_{a_3b_2}\cup l_{a_4b_3}$ or $b_5\notin l_{a_3b_3}\cup l_{a_4b_2}$ (indeed, suppose that $b_5\in l_{a_i,b_j}$, then if $i'=i$ or $j'=j$, $b_5\notin l_{a_{i'}b_{j'}}$) since no five points in $P$ are collinear). So, we may isolate $A$ using one of these pairs of lines and a line intersecting only $b_4$. 

It remains to consider the case that $|A\cap \{a_1,...,a_4,b_1,b_2,b_3\}|=3$. We may assume without loss of generality that $a_3,a_4,b_3,b_4,b_5 \in A$ and that $a_1,a_2,b_1,b_2\notin A$. 

Now given (\ref{O3line-intersect-3collin}), we see that whenever $k\in O_3(l)$, $k\in O_3(l')$ where $l'$ is the line passing through $a_1,b_1,b_2,b_3$. This implies that the following is true.
\begin{itemize}
    \item[\namedlabel{star}{($\star$)}] Whenever $k\in O_{\geq3}(l')$, $k=l$ or $k\in O_{\geq3}(l)$ (and vice versa reversing $l$ and $l'$).
\end{itemize}

Now for $i\in \{3,4\}$ and $j\in \{4,5\}$, if $a_i$ does not pair with $b_j$ inside $A$, then $a_i,b_j$ must lie on a line that is in $O_{\geq3}(l)$. Hence, by observation \ref{star}, $a_i,b_j$ must lie on a line that is in $O_{\geq3}(l')$. So, by \Cref{cond:d}, since $l\in O_4(l')$, for each $i\in \{3,4\}$ there is $j_i\in \{4,5\}$ such that $a_i$ pairs with $b_{j_i}$ in $A$. If we may pick $j_1\neq j_2$, we may isolate $A$ using $l_{a_1b_{j_1}}, l_{a_2b_{j_2}}$ and a line intersecting only $b_3$. 

So suppose that we cannot. Then we may assume without loss of generality that $b_5$ does not pair inside $A$ with either $a_3$ or $a_4$ (if both $b_4,b_5$ do pair inside $A$ with one of $a_3, a_4$, then we may pick $j_1, j_2$ distinct). Note that in this case $a_3,a_4$ both pair with $b_4$ inside $A$. Now, $b_5$ is on two distinct lines $l_3,l_4\in O_{\geq3}(l)$ where $a_3\in l_3$ and $a_4 \in l_4$, so by \Cref{cond:d} and \ref{star}, we must have either that $b_5$ pairs with $b_3$ in $A$ or $b_3\in l_3\cup l_4$. But in the latter case there are three collinear points in $A$, so we have that $b_5$ pairs with $b_3$ inside $A$. Therefore, we may isolate $A$ using $l_{b_3b_5}, l_{a_3b_4}$ and a line intersecting only $a_4$.  

\noindent\textbf{Case II:} Suppose that $|A|=6$ and there is a unique set of four collinear points in $P$ , $\{a_1,...,a_4\}$. Denote the remainder by $\{b_1,...b_5\}$, assuming as before that $b_1,b_2,b_3$ are collinear and that neither of $b_4,b_5$ are collinear with these three. Then since no three points in $A$ are collinear, we may suppose without loss of generality that $A=\{a_1,a_2,b_1,b_2,b_4,b_5\}$. By condition (\ref{cond:b}) we have that some $b_i,b_j$ pair inside $A$ for $i,j\in \{1,2,4,5\}$. Since we may not have that $i,j\in \{1,2\}$, we need consider only the following subcases.

\textbf{Subcase II(a):} Suppose $b_4,b_5$ pair inside $A$. Then we note that if $m,n\in \{1,2\}$, $a_m,b_n$ pair inside $A$ since there are no five points collinear and a unique set of four points collinear. This will hold in general as long as $b_4,b_5\in A$ and future use of this fact in the proof will be referred to by \namedlabel{starstar}{$(\star\star)$}. So, $a_1,b_1$ and $a_2,b_2$ pair inside $A$ and we may isolate $A$ using $l_{a_1b_1},l_{a_2b_2},l_{b_4b_5}$. 

\textbf{Subcase II(b):} Suppose $b_4,b_5$ do not pair inside $A$, then without loss of generality we may suppose that $b_4,b_1$ pair inside $A$. Then note that $b_5$ may fail to pair with $a_1$ (respectively, $a_2$) in $A$ only if $b_5,b_3,a_1$ (respectively, $b_5,b_3,a_2$) collinear. So, since there are no five points collinear in $A$, $b_5$ must pair with either $a_1$ or $a_2$. So, without loss of generality we may assume that $a_1,b_5$ pair inside $A$. By \ref{starstar}, $a_2,b_2$ pair inside $A$, so may isolate $A$ using $l_{a_1b_5},l_{a_2b_2},l_{b_1b_4}$.

\noindent\textbf{Case III:} Then if $|A|=5$, keeping the notation from Case II and still with a unique set of four collinear points, we have the following subcases.

\textbf{Subcase III(a):} Suppose $|\{b_4,b_5\}\cap A|=1$, so without loss of generality we have that $A=\{a_1,a_2,b_1,b_2,b_4\}$. Now if $a_1,b_1$ and $a_2,b_2$ each pair inside $A$, we are done, so we may suppose otherwise. Without loss of generality, we may suppose that $a_1,b_1$ do not pair inside $A$. So, since there is a unique set of four points collinear, we may suppose that $b_5\in l_{a_1b_1}$. But then, again since there is a unique set of four points collinear, $b_5\notin l_{a_1b_2}$ and $b_5\notin l_{a_2b_1}$. So, we may isolate $A$ in this case too. 

\textbf{Subcase III(b):} Suppose $|\{b_1,b_2,b_3\}\cap A|=1$, so we may assume without loss of generality that $A=\{a_1,a_2,b_1,b_4,b_5\}$. By (\ref{nobadquadruple}) $\{a_1,a_2,b_4,b_5\}$ is not a bad quadruple with respect to $\{a_1,a_2,a_3,a_4\}$ and $\{b_1,b_4,b_5\}$. So, we may assume that (without loss of generality) $a_1,b_4$ pair inside $A$. Further, $b_1,a_2$ must pair inside $A$ by \ref{starstar}, so we may isolate $A$ using $l_{a_1b_4},l_{a_2b_1}$ and a line intersecting only $b_5$.

\textbf{Subcase III(c):} Suppose $|\{a_1,a_2,a_3,a_4\}\cap A|=1$, so that without loss of generality we may assume that $A=\{a_1,b_1,b_2,b_4,b_5\}$. By condition (\ref{cond:b}), we may assume that some $b_i,b_j$ for distinct $i,j\in \{1,2,4,5\}$ pairs inside $A$. But since it can't be that $i,j\in \{1,2\}$ we may without loss of generality suppose that $1\neq i,j$. But, by \ref{starstar}, $b_1,a_1$ pair inside $A$. Thus, we may isolate $A$ using $l_{a_1b_1},l_{b_ib_j}$ and a line through the remaining point.

These cases are evidently exhaustive.
\end{proof}

\begin{proof}[Proof of (\ref{3collin-leq6})]
Suppose initially that $|A|=6$.

    We split cases according to whether the following property holds or not. Set
   $$T: \text{ the line through $a_1, a_2, a_3$ contains another point of $P\setminus A$, say $q_1$.}$$

\noindent \textbf{Case I:} $T$ holds. Then we need to find a matching of $\{a_1, a_2, a_3\}$ and $\{b_1, b_2, b_3\}$ in $P$.
By (\ref{nobadquadruple}), there is no bad quadruple. So, by \Cref{3smatch-withconds} there is a matching of $\{a_1, a_2, a_3\}$ and $\{b_1, b_2, b_3\}$ in $P \setminus \{q_1\}$. But evidently such a matching is a matching in $P$ as well, since $q_1$ collinear with $a_1,a_2,a_3$.

\noindent \textbf{Case II:} $T$ does not hold. In this case, we observe that all points in $P\setminus \{a_1, a_2, a_3\}$ are contained in two lines. Indeed, if not, then by (\ref{cond:0}) we would need one line to contain three points among $P\setminus \{a_1, a_2, a_3\}$ and one of the elements in $\{a_1, a_2, a_3\}$, making the line through $\{a_1, a_2, a_3\}$ into an ordinary 3-line, contradicting (\ref{set1-4collin}).

\textbf{Subcase II(a):} Every $l_{b_i b_j}$ contains some $q_k$. In this case, we claim that $b_1, b_2, b_3$ are collinear. Indeed, since $\{b_1, b_2, b_3, q_1, q_2, q_3\}$ are contained in two lines, any line passing through three points from this set will be among those two lines. $l_{b_1 b_2}$, $l_{b_1 b_3}$ and $l_{b_1 b_2}$ are three such lines. Thus, they may not be pairwise distinct. So, $b_1, b_2, b_3$ are collinear.

Now, we are in a situation where $b_1, b_2, b_3$ are three collinear points in $A$, such that property $T$ holds for them in place of $a_1, a_2, a_3$. By swapping their role, we are done by Case I.

\textbf{Subcase II(b):} Some $l_{b_i b_j}$ does not contain any $q_k$. In this case, we can use the line through the $a_1, a_2, a_3$. We are left with three points, $b_1, b_2, b_3$, which we need to isolate from a set of six points with two lines. Take also the line $l_{b_i b_j}$. We are left with one point only and we can use the final line.

It remains to consider the case section $|A|=5$. Now, if we have a set $A$ with $|A|=5$ and a set of 3 points in $A$ collinear, $\{a_1,a_2,a_3\}$ (letting $A:=\{a_1,a_2,a_3,b_1,b_2\}$ and $P\setminus A:= \{q_1,...,q_4\}$), we see that $A\cup \{q_4\}$ satisfies the original conditions. We therefore have a matching of $\{a_1,a_2,a_3\}$ and $\{b_1,b_2,q_4\}$ in $P$. So we may isolate $A$ using the lines in the matching containing $b_1$ and $b_2$, and a final line intersecting only the remaining point in $\{a_1,a_2,a_3\}$.
\end{proof}

\newpage
\section{Classification of Small Shatterable Sets}\label{app:pictures}
We conclude our paper by giving a somewhat more concrete version of the results in \Cref{sec:counting}. Indeed, in this section, we simply provide representatives for each of the isomorphism types of sets of maximal VC-dimension that can be shattered by $\F_2$ and $\F_3$.
\subsection{Sets Shattered by \texorpdfstring{$\F_2$}{F2}}
We exhibit the following sets that satisfy each subcase in the proof of \Cref{iso2}. 

\noindent\textbf{I:}
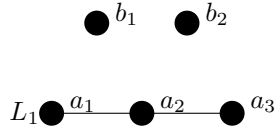
\begin{center}
\begin{tikzpicture}[scale=0.6]
\tikzset{dot/.style={fill=black,circle}}

{
\draw (0,0) -- (4,0);
\node at (-0.6,0) {$L_1$};
}

\node[above, right] at (0.2,0.2) {$a_1$};
\node[dot] at (0,0){};
\node[above, right] at (2.2,0.2) {$a_2$};
\node[dot] at (2,0){};
\node[above, right] at (4.2,0.2) {$a_3$};
\node[dot] at (4,0){};
\node[above, right] at (1.2,2.2) {$b_1$};
\node[dot] at (1,2){};
\node[above, right] at (3.2,2.2) {$b_2$};
\node[dot] at (3,2){};

\end{tikzpicture}      
\end{center}

\noindent\textbf{II:}
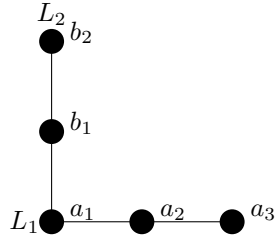
\begin{center}
\begin{tikzpicture}[scale=0.6]
\tikzset{dot/.style={fill=black,circle}}

{
\draw (0,0) -- (4,0);
\node at (-0.6,0) {$L_1$};
\draw (0,0) -- (0,4);
\node at (0,4.6) {$L_2$};
}

\node[above, right] at (0.2,0.2) {$a_1$};
\node[dot] at (0,0){};
\node[above, right] at (2.2,0.2) {$a_2$};
\node[dot] at (2,0){};
\node[above, right] at (4.2,0.2) {$a_3$};
\node[dot] at (4,0){};
\node[above, right] at (.2,2.2) {$b_1$};
\node[dot] at (0,2){};
\node[above, right] at (0.2,4.2) {$b_2$};
\node[dot] at (0,4){};

\end{tikzpicture}      
\end{center}

\subsection{Sets Shattered by \texorpdfstring{$\F_3$}{F3}}
We exhibit the following sets that satisfy each subcase in the proof of \Cref{iso3} - one may check using \Cref{3linesleft} and \Cref{3linesleft'} that each of these sets can, indeed be shattered by $\F_3$. 

\noindent\textbf{I(a):}
\begin{center}
\begin{tikzpicture}
\tikzset{dot/.style={fill=black,circle}}

{
\draw[red] (4,3) -- (8,3);
\node at (2.6,3) {$L_2$};
\draw[red] (2,2) -- (8,2);
\node at (0.6,2) {$L_1$};
\draw (2,2) -- (6,4);
\node at (2,1.4) {$l_1$};
\draw (4,2) -- (4,4);
\node at (4,1.4) {$l_2$};
\draw (6,2) -- (6,4);
\node at (6,1.4) {$l_3$};
\draw (8,2) -- (4,4);
\node at (8,1.4) {$l_4$};
\draw[red] (4,4) -- (6,4);
}

\node[above, right] at (2.2,2.2) {$b_1$};
\node[dot] at (2,2){};
\node[above, right] at (4.2,2.2) {$b_2$};
\node[dot] at (4,2){};
\node[above, right] at (6.2,2.2) {$b_3$};
\node[dot] at (6,2){};
\node[above, right] at (8.2,2.2) {$b_4$};
\node[dot] at (8,2){};
\node[above, right] at (3.4,3.2) {$a_1$};
\node[dot] at (4,3){};
\node[above, right] at (6.2,3.2) {$a_2$};
\node[dot] at (6,3){};
\node[above, right] at (8.2,3.2) {$a_3$};
\node[dot] at (8,3){};
\node[above, right] at (4.2,4.2) {$a_4$};
\node[dot] at (4,4){};
\node[above, right] at (6.2,4.2) {$a_5$};
\node[dot] at (6,4){};

\end{tikzpicture}      
\end{center}
\noindent\textbf{I(b):}
\begin{center}

\begin{tikzpicture}
\tikzset{dot/.style={fill=black,circle}}

{
\draw[red] (3,3) -- (7,3);
\node at (1.6,3) {$L_2$};
\draw[red] (2,2) -- (8,2);
\node at (0.6,2) {$L_1$};
\draw (2,2) -- (4,4);
\node at (2,1.4) {$l_1$};
\draw (4,2) -- (6,4);
\node at (4,1.4) {$l_2$};
\draw (6,2) -- (4,4);
\node at (6,1.4) {$l_3$};
\draw (8,2) -- (6,4);
\node at (8,1.4) {$l_4$};
\draw[red] (4,4) -- (6,4);
}

\node[above, right] at (2.3,2.2) {$b_1$};
\node[dot] at (2,2){};
\node[above, right] at (4.3,2.2) {$b_2$};
\node[dot] at (4,2){};
\node[above, right] at (6.3,2.2) {$b_3$};
\node[dot] at (6,2){};
\node[above, right] at (8.3,2.2) {$b_4$};
\node[dot] at (8,2){};
\node[above, right] at (3.3,3.2) {$a_1$};
\node[dot] at (3,3){};
\node[above, right] at (5.3,3.2) {$a_2$};
\node[dot] at (5,3){};
\node[above, right] at (7.3,3.2) {$a_3$};
\node[dot] at (7,3){};
\node[above, right] at (4.3,4.2) {$a_4$};
\node[dot] at (4,4){};
\node[above, right] at (6.3,4.2) {$a_5$};
\node[dot] at (6,4){};

\end{tikzpicture}      
\end{center}
\noindent\textbf{II(a):}
\begin{center}
\begin{tikzpicture}[scale=0.6]
\tikzset{dot/.style={fill=black,circle}}

{
\draw[red] (2,4) -- (11,7);
\draw[red] (2,4) -- (11,1);
\draw (5,5) -- (8,2);
\draw (5,3) -- (8,6);
\draw (11,7) -- (11,1);
\draw[red] (6,4) -- (11,5);
}
\node[above, right] at (2.2,4) {$b_1$};
\node[dot] at (2,4){};
\node[above, right] at (6.2,4) {$a_4$};
\node[dot] at (6,4){};
\node[above, right] at (5.2,5.4) {$a_1$};
\node[dot] at (5,5){};
\node[above, right] at (8.2,6.4) {$a_2$};
\node[dot] at (8,6){};
\node[above, right] at (11.2,7.2) {$a_3$};
\node[dot] at (11,7){};
\node[above, right] at (5.2,3.1) {$b_2$};
\node[dot] at (5,3){};
\node[above, right] at (8.2,2.2) {$b_3$};
\node[dot] at (8,2){};
\node[above, right] at (11.2,1.2) {$b_4$};
\node[dot] at (11,1){};
\node[above, right] at (11.2,5.2) {$a_5$};
\node[dot] at (11,5){};
\end{tikzpicture}
\end{center}

\noindent\textbf{II(b):}
\begin{center}
\begin{tikzpicture}
\tikzset{dot/.style={fill=black,circle}}

{
\draw[red] (2,1) -- (2,7);
\draw[red] (2,1) -- (8,1);
\draw[red] (2,5) -- (6,1);
\draw (2,7) -- (4,1);
\draw (8,1) -- (2,3);
}
\node[above, right] at (2.1,1.2) {$a_1$};
\node[dot] at (2,1){};
\node[above, right] at (2.1,3.2) {$b_1$};
\node[dot] at (2,3){};
\node[above, right] at (2.1,5.2) {$b_2$};
\node[dot] at (2,5){};
\node[above, right] at (2.1,7.2) {$b_3$};
\node[dot] at (2,7){};
\node[above, right] at (4.2,1.2) {$a_2$};
\node[dot] at (4,1){};
\node[above, right] at (6.2,1.2) {$a_3$};
\node[dot] at (6,1){};
\node[above, right] at (8.2,1.2) {$a_4$};
\node[dot] at (8,1){};
\node[above, right] at (3.2,4.2) {$b_4$};
\node[dot] at (3,4){};
\node[above, right] at (5.2,2.2) {$b_5$};
\node[dot] at (5,2){};
\end{tikzpicture}
\end{center}

\noindent\textbf{III:}
\begin{center}
\begin{tikzpicture}[scale=0.6]
\tikzset{dot/.style={fill=black,circle}}

{
\draw[red] (0,0) -- (6,0);
\draw[red] (1.5,3) -- (6,3);
\draw (3,0) -- (3,6);
\draw (0,0) -- (3,6);
\draw (0,0) -- (6,3);
\draw (3,3) -- (6,0);
\draw (6,3) -- (6,-6);
\draw (6,-6) -- (1.5,3);
\draw[red] (6,-6) -- (3,6);
}
\node[above, right] at (0.2,0.2) {$a_1$};
\node[dot] at (0,0){};
\node[above, right] at (3.2,0.2) {$a_2$};
\node[dot] at (3,0){};
\node[above, right] at (6.2,0.2) {$a_3$};
\node[dot] at (6,0){};
\node[above, right] at (1.7,3.2) {$b_1$};
\node[dot] at (1.5,3){};
\node[above, right] at (3.2,3.2) {$b_2$};
\node[dot] at (3,3){};
\node[above, right] at (6.2,3.2) {$b_3$};
\node[dot] at (6,3){};
\node[above, right] at (3.2,6.2) {$c_1$};
\node[dot] at (3,6){};
\node[above, right] at (6.2,-6.2) {$c_3$};
\node[dot] at (6,-6){};
\node[above, right] at (4.2,2.2) {$c_2$};
\node[dot] at (4,2){};
\end{tikzpicture}
\end{center}

\bibliographystyle{alpha}
\bibliography{bibliography}

\end{document}